\documentclass[fleqn,preprint,3p,a4paper]{elsarticle}

\usepackage{amssymb,mathrsfs,latexsym,amsfonts}
\usepackage{amsmath}
\usepackage{amsthm}
\usepackage{dcolumn}
\usepackage{endnotes}
\usepackage{tabularx}
\usepackage[matrix,arrow]{xy}
\usepackage{wasysym}

\newtheorem{theorem}{Theorem}[section]
\newtheorem{proposition}[theorem]{Proposition}
\newtheorem{lemma}[theorem]{Lemma}
\newtheorem{corollary}[theorem]{Corollary}

\theoremstyle{definition}
\newtheorem{definition}[theorem]{Definition}
\newtheorem{example}[theorem]{Example}

\newtheorem{question}[theorem]{Question}

\journal{Topology and its Applications}

\begin{document}

\begin{frontmatter}

%% Title, authors and addresses

%% use the tnoteref command within \title for footnotes;
%% use the tnotetext command for theassociated footnote;
%% use the fnref command within \author or \address for footnotes;
%% use the fntext command for theassociated footnote;
%% use the corref command within \author for corresponding author footnotes;
%% use the cortext command for theassociated footnote;
%% use the ead command for the email address,
%% and the form \ead[url] for the home page:
%% \title{Title\tnoteref{label1}}
%% \tnotetext[label1]{}
%% \author{Name\corref{cor1}\fnref{label2}}
%% \ead{email address}
%% \ead[url]{home page}
%% \fntext[label2]{}
%% \cortext[cor1]{}
%% \address{Address\fnref{label3}}
%% \fntext[label3]{}

\title{Quotient spaces with strong subgyrogroups \tnoteref{t1}}
\tnotetext[t1]{This research was supported by the National Natural Science Foundation of China (Nos. 12071199, 11661057), the Natural Science Foundation of Jiangxi Province, China (No. 20192ACBL20045).}

%% use optional labels to link authors explicitly to addresses:
%% \author[label1,label2]{}
%% \address[label1]{}
%% \address[label2]{}

\author[M. Bao]{Meng Bao}
\ead{mengbao95213@163.com}
\address[M. Bao]{College of Mathematics, Sichuan University, Chengdu 610064, China}

\author[X. Ling]{Xuewei Ling}
\ead{781736783@qq.com}
\address[X. Ling]{Institute of Mathematics, Nanjing Normal University, Nanjing 210046, China}

\author[X. Xu]{Xiaoquan Xu\corref{mycorrespondingauthor}}
\cortext[mycorrespondingauthor]{Corresponding author.}
\ead{xiqxu2002@163.com}
\address[X. Xu]{Fujian Key Laboratory of Granular Computing and Applications, Minnan Normal University, Zhangzhou 363000, China}

\begin{abstract}
In this paper, we mainly investigate the quotient spaces $G/H$ when $G$ is a strongly topological gyrogroup and $H$ is a strong subgyrogroup of $G$. It is shown that if $G$ is a strongly topological gyrogroup, $H$ is a closed strong subgyrogroup of $G$ and $H$ is inner neutral, then the quotient space $G/H$ is first-countable if and only if $G/H$ is a bisequential space if and only if $G/H$ is a weakly first-countable space if and only if $G/H$ is a $csf$-countable and sequential $\alpha_{7}$-space. Moreover, it is shown that if $G$ is a strongly topological gyrogroup and $H$ is a locally compact strong subgyrogroup of $G$, then there exists an open neighborhood $U$ of the identity element $0$ such that $\pi (\overline{U})$ is closed in $G/H$ and the restriction of $\pi $ to $\overline{U}$ is a perfect mapping from $\overline{U}$ onto the subspace $\pi (\overline{U})$; if $H$ is a locally compact metrizable strong subgyrogroup of $G$ and the quotient space $G/H$ is sequential, then $G$ is also sequential; if $H$ is a closed first-countable and separable strong subgyrogroup of $G$, the quotient space $G/H$ is an $\aleph_{0}$-space, then $G$ is an $\aleph_{0}$-space; if the quotient space $G/H$ is a cosmic space, then $G$ is also a cosmic space; if the quotient space $G/H$ has a star-countable $cs$-network or star-countable $wcs^{*}$-network, then $G$ also has a star-countable $cs$-network or star-countable $wcs^{*}$-network, respectively.

\end{abstract}


%%Graphical abstract
%\begin{graphicalabstract}
%\includegraphics{grabs}
%\end{graphicalabstract}

%%Research highlights
%\begin{highlights}
%\item Research highlight 1
%\item Research highlight 2
%\end{highlights}

%% MSC codes here, in the form:
%% or \MSC[2008] code \sep code (2000 is the default)

\end{frontmatter}

%% \linenumbers

%% main text
\section{Introduction}

Let $c$ be a positive constant representing the speed of light in vacuum and $\mathbb{R}_{c}^{3}=\{\mathbf{v}\in \mathbb{R}^{3}:||\mathbf{v}||<c\}$ the $c$-ball of relativistically admissible velocities with Einstein velocity addition. It is well-known that Einstein velocity addition is the standard velocity addition of relativistically admissible velocities that Einstein introduced in 1905 that founded the special theory of relativity. The Einstein velocity addition $\oplus _{E}$ is given as the following: $$\mathbf{u}\oplus _{E}\mathbf{v}=\frac{1}{1+\frac{\mathbf{u}\cdot \mathbf{v}}{c^{2}}}(\mathbf{u}+\frac{1}{\gamma _{\mathbf{u}}}\mathbf{v}+\frac{1}{c^{2}}\frac{\gamma _{\mathbf{u}}}{1+\gamma _{\mathbf{u}}}(\mathbf{u}\cdot \mathbf{v})\mathbf{u}),$$ where $\mathbf{u,v}\in \mathbb{R}_{c}^{3}$ and $\gamma _{\mathbf{u}}$ is given by $$\gamma _{\mathbf{u}}=\frac{1}{\sqrt{1-\frac{\mathbf{u}\cdot \mathbf{u}}{c^{2}}}}.$$ By the study of the $c$-ball of relativistically admissible velocities, A.A. Ungar discovered that the seemingly structureless Einstein addition of relativistically admissible velocities possesses a rich grouplike structure and he posed the concept of gyrogroups in \cite{UA2002}. Indeed, a gyrogroup is the most natural extension of a group into the regime of the nonassociative algebra that we need for extending analytic Euclidean geometry into analytic hyperbolic geometry. In 2017, the gyrogroup was equipped with a topology by W. Atiponrat \cite{AW} such that the binary operation $\oplus: G\times G\rightarrow G$ is jointly continuous and the inverse mapping $\ominus (\cdot): G\rightarrow G$, i.e. $x\rightarrow \ominus x$, is also continuous and she called it topological gyrogroups. At the same time, she asked whether the first-countability axiom can imply metrizability in a Hausdorff topological gyrogroup?  Shortly afterwards, Cai, Lin and He in \cite{CZ} proved that every topological gyrogroup is a rectifiable space, which implies that the equivalence between the axioms of first-countable and metrizability in a Hausdorff topological gyrogroup. By further researches of the classical M\"{o}bius gyrogroups, Einstein gyrogroups, and Proper Velocity gyrogroups, Bao and Lin \cite{BL} found that each of them has an open neighborhood base at the identity element $0$ such that all elements of the base are invariant under the groupoid automorphisms with standard topology. Therefore, they posed the concept of strongly topological gyrogroups and showed that every feathered strongly topological gyrogroup is paracompact. A series of results on topological gyrogroups and strongly topological gyrogroups have been obtained in \cite{AA2010,AW1,AW2020,BL1,BL2,BL3,BLL,BLX,BZX,BZX2,LF,LF1,LF2}. In particular, it was proved in \cite{BL1,BL2} that each $T_{0}$-strongly topological gyrogroup is completely regular, every $T_{0}$-strongly topological gyrogroup with a countable pseudocharacter is submetrizable and each locally paracompact strongly topological gyrogroup is paracompact.

In \cite{ST}, T. Suksumran and K. Wiboonton posed the notion of $L$-subgyrogroups and showed that if $H$ is an $L$-subgyrogroup of a gyrogroup $G$, then the set $\{a\oplus H:a\in G\}$ forms a disjoint partition of $G$. Therefore, it is natural to research the quotient spaces of a strongly topological gyrogroup with respect to $L$-subgyrogroups as left cosets. Moreover, Bao and Lin proved that if $G$ is a $T_{0}$-strongly topological gyrogroup with a symmetric neighborhood base $\mathscr U$ at $0$ and $H$ is an admissible subgyrogroup generated from $\mathscr U$, then the left coset space $G/H$ is submetrizable. Recently, Bao and Xu \cite{BX2022} constructed a subgyrogroup $H$ in a strongly topological gyrogroup $G$ such that $gyr[x,y](H)=H$ for all $x,y\in G$, hence they introduced the concept of strong subgyrogroups in a strongly topological gyrogroup. They showed that if $G$ is a strongly topological gyrogroup, $H$ is a closed strong subgyrogroup of $G$ and $H$ is inner neutral, then $G/H$ is first-countable if and only if $G/H$ is Fr\'echet-Urysohn with an $\omega^{\omega}$-base.

In this paper, we continue to research the quotient spaces of strongly topological gyrogroups with respect to closed strong subgyrogroups. First, it is not difficult to see that if $G$ is a strongly topological gyrogroup with a symmetric neighborhood base $\mathscr U$ at $0$, $P$ is an admissible subgyrogroup generated from $\mathscr U$ and the quotient space $G/P$ is first-countable, then it is metrizable. However, the two arrows space is a compact coset space which is first-countable, but not submetrizable, and it is homeomorphic to $G/H$, where $G$ is a strongly topological gyrogroup and $H$ is a closed strong subgyrogroup of $G$. Moreover, it is shown that if $G$ is a strongly topological gyrogroup, $H$ is a closed strong subgyrogroup of $G$ and $H$ is inner neutral, then the quotient space $G/H$ is first-countable if and only if $G/H$ is a bisequential space if and only if $G/H$ is a weakly first-countable space if and only if $G/H$ is a $csf$-countable and sequential $\alpha_{7}$-space. More important, it is shown that if $G$ is a strongly topological gyrogroup and $H$ is a locally compact strong subgyrogroup of $G$, then there exists an open neighborhood $U$ of the identity element $0$ such that $\pi (\overline{U})$ is closed in $G/H$ and the restriction of $\pi $ to $\overline{U}$ is a perfect mapping from $\overline{U}$ onto the subspace $\pi (\overline{U})$, which implies that if $H$ is a locally compact strong subgyrogroup of a strongly topological gyrogroup $G$ and the quotient space $G/H$ has some nice properties, such as locally compact, locally countably compact, locally pseudocompact, etc., then $G$ also has the same properties. Then, we show that when $G$ is a strongly topological gyrogroup and $H$ is a locally compact metrizable strong subgyrogroup of $G$, if the quotient space $G/H$ is strictly (strongly) Fr\'echet-Urysohn, then $G$ is also strictly (strongly) Fr\'echet-Urysohn; if the quotient space $G/H$ is sequential, then $G$ is also sequential. Therefore, some important and interesting results in \cite{LS} are improved. Finally, we study the quotient space $G/H$ with some generalized metric properties, where $G$ is a strongly topological gyrogroup and $H$ is a closed first-countable and separable strong subgyrogroup of $G$. We prove that if the quotient space $G/H$ is an $\aleph_{0}$-space, then $G$ is an $\aleph_{0}$-space; if the quotient space $G/H$ is a cosmic space, then $G$ is also a cosmic space; if the quotient space $G/H$ has a star-countable $cs$-network or star-countable $wcs^{*}$-network, then $G$ also has a star-countable $cs$-network or star-countable $wcs^{*}$-network, respectively.

\section{Preliminary}

Throughout this paper, all topological spaces are assumed to be Hausdorff, unless otherwise is explicitly stated. Let $\mathbb{N}$ be the set of all positive integers and $\omega$ the first infinite ordinal. The readers may consult \cite{AA, E, linbook, UA} for notation and terminology not explicitly given here. Next we recall some definitions and facts.

\begin{definition}\cite{AW}
Let $G$ be a nonempty set, and let $\oplus: G\times G\rightarrow G$ be a binary operation on $G$. Then the pair $(G, \oplus)$ is called a {\it groupoid}. A function $f$ from a groupoid $(G_{1}, \oplus_{1})$ to a groupoid $(G_{2}, \oplus_{2})$ is called a {\it groupoid homomorphism} if $f(x\oplus_{1}y)=f(x)\oplus_{2} f(y)$ for any elements $x, y\in G_{1}$. Furthermore, a bijective groupoid homomorphism from a groupoid $(G, \oplus)$ to itself will be called a {\it groupoid automorphism}. We write $\mbox{Aut}(G, \oplus)$ for the set of all automorphisms of a groupoid $(G, \oplus)$.
\end{definition}

\begin{definition}\cite{UA}
Let $(G, \oplus)$ be a groupoid. The system $(G,\oplus)$ is called a {\it gyrogroup}, if its binary operation satisfies the following conditions:

\smallskip
$(G1)$ There exists a unique identity element $0\in G$ such that $0\oplus a=a=a\oplus0$ for all $a\in G$.

\smallskip
$(G2)$ For each $x\in G$, there exists a unique inverse element $\ominus x\in G$ such that $\ominus x \oplus x=0=x\oplus (\ominus x)$.

\smallskip
$(G3)$ For all $x, y\in G$, there exists $\mbox{gyr}[x, y]\in \mbox{Aut}(G, \oplus)$ with the property that $x\oplus (y\oplus z)=(x\oplus y)\oplus \mbox{gyr}[x, y](z)$ for all $z\in G$.

\smallskip
$(G4)$ For any $x, y\in G$, $\mbox{gyr}[x\oplus y, y]=\mbox{gyr}[x, y]$.
\end{definition}

\begin{lemma}\cite{UA}\label{a}
Let $(G, \oplus)$ be a gyrogroup. Then for any $x, y, z\in G$, we obtain the followings:

\begin{enumerate}
\smallskip
\item $(\ominus x)\oplus (x\oplus y)=y$. \ \ \ (left cancellation law)

\smallskip
\item $(x\oplus (\ominus y))\oplus \mbox{gyr}[x, \ominus y](y)=x$. \ \ \ (right cancellation law)

\smallskip
\item $(x\oplus \mbox{gyr}[x, y](\ominus y))\oplus y=x$.

\smallskip
\item $\mbox{gyr}[x, y](z)=\ominus (x\oplus y)\oplus (x\oplus (y\oplus z))$.

\smallskip
\item $(x\oplus y)\oplus z=x\oplus (y\oplus \mbox{gyr}[y,x](z))$.
\end{enumerate}
\end{lemma}

\begin{proposition}\cite{ST}\label{2mt1}
Let $G$ be a gyrogroup and let $X\subseteq G$. Then the followings are equivalent:

(1) $\mbox{gyr}[a,b](X)\subseteq X$ for all $a,b\in G$;

(2) $\mbox{gyr}[a,b](X)=X$ for all $a,b\in G$.
\end{proposition}

Notice that a group is a gyrogroup $(G,\oplus)$ such that $\mbox{gyr}[x,y]$ is the identity function for all $x, y\in G$. The definition of a subgyrogroup is as follows.

\begin{definition}\cite{ST}
Let $(G,\oplus)$ be a gyrogroup. A nonempty subset $H$ of $G$ is called a {\it subgyrogroup}, denoted
by $H\leq G$, if $H$ forms a gyrogroup under the operation inherited from $G$ and the restriction of $\mbox{gyr}[a,b]$ to $H$ is an automorphism of $H$ for all $a,b\in H$.

\smallskip
Furthermore, a subgyrogroup $H$ of $G$ is said to be an {\it $L$-subgyrogroup}, denoted
by $H\leq_{L} G$, if $\mbox{gyr}[a, h](H)=H$ for all $a\in G$ and $h\in H$.
\end{definition}

\begin{definition}\cite{AW}
A triple $(G, \tau, \oplus)$ is called a {\it topological gyrogroup} if the following statements hold:

\smallskip
(1) $(G, \tau)$ is a topological space.

\smallskip
(2) $(G, \oplus)$ is a gyrogroup.

\smallskip
(3) The binary operation $\oplus: G\times G\rightarrow G$ is jointly continuous while $G\times G$ is endowed with the product topology, and the operation of taking the inverse $\ominus (\cdot): G\rightarrow G$, i.e. $x\rightarrow \ominus x$, is also continuous.
\end{definition}

Obviously, every topological group is a topological gyrogroup. However, every topological gyrogroup whose gyrations are not identically equal to the identity is not a topological group. In particular, it was proved in \cite{AW} that the Einstein gyrogroup with the standard topology is a topological gyrogroup but not a topological group. Next, we introduce the definition of a strongly topological gyrogroup, it is very important in this paper.

\begin{definition}{\rm (\cite{BL})}\label{d11}
Let $G$ be a topological gyrogroup. We say that $G$ is a {\it strongly topological gyrogroup} if there exists a neighborhood base $\mathscr U$ of $0$ such that, for every $U\in \mathscr U$, $\mbox{gyr}[x, y](U)=U$ for any $x, y\in G$. For convenience, we say that $G$ is a strongly topological gyrogroup with neighborhood base $\mathscr U$ of $0$.
\end{definition}

For each $U\in \mathscr U$, we can set $V=U\cup (\ominus U)$. Then, $$\mbox{gyr}[x,y](V)=\mbox{gyr}[x, y](U\cup (\ominus U))=\mbox{gyr}[x, y](U)\cup (\ominus \mbox{gyr}[x, y](U))=U\cup (\ominus U)=V,$$ for all $x, y\in G$. Obviously, the family $\{U\cup(\ominus U): U\in \mathscr U\}$ is also a neighborhood base of $0$. Therefore, we may assume that $U$ is symmetric for each $U\in\mathscr U$ in Definition~\ref{d11}. Moreover, in the classical M\"{o}bius, Einstein, or Proper Velocity gyrogroups, we know that gyrations are indeed special rotations, however for an arbitrary gyrogroup, gyrations belong to the automorphism group of $G$ and need not be necessarily rotations.

In \cite{BL}, the authors proved that there is a strongly topological gyrogroup which is not a topological group, see Example \ref{lz1}.

\begin{example}\cite{BL}\label{lz1}
Let $\mathbb{D}$ be the complex open unit disk $\{z\in \mathbb{C}:|z|<1\}$. We consider $\mathbb{D}$ with the standard topology. In \cite[Example 2]{AW}, define a M\"{o}bius addition $\oplus _{M}: \mathbb{D}\times \mathbb{D}\rightarrow \mathbb{D}$ to be a function such that $$a\oplus _{M}b=\frac{a+b}{1+\bar{a}b}\ \mbox{for all}\ a, b\in \mathbb{D}.$$ Then $(\mathbb{D}, \oplus _{M})$ is a gyrogroup, and it follows from \cite[Example 2]{AW} that $$\mbox{gyr}[a, b](c)=\frac{1+a\bar{b}}{1+\bar{a}b}c\ \mbox{for any}\ a, b, c\in \mathbb{D}.$$ For any $n\in \mathbb{N}$, let $U_{n}=\{x\in \mathbb{D}: |x|\leq \frac{1}{n}\}$. Then, $\mathscr U=\{U_{n}: n\in \mathbb{N}\}$ is a neighborhood base of $0$. Moreover, we observe that $|\frac{1+a\bar{b}}{1+\bar{a}b}|=1$. Therefore, we obtain that $\mbox{gyr}[x, y](U)\subseteq U$, for any $x, y\in \mathbb{D}$ and each $U\in \mathscr U$, then it follows that $\mbox{gyr}[x, y](U)=U$ by Proposition \ref{2mt1}. Hence, $(\mathbb{D}, \oplus _{M})$ is a strongly topological gyrogroup. However, $(\mathbb{D}, \oplus _{M})$ is not a group by \cite[Example 2]{AW}.
\end{example}

Then, Bao and Lin constructed the following example by Example \ref{lz1} to show that for any cardinality $\kappa>\omega$, there exists a gyrogroup $G$ with subgyrogroup $H$ of the cardinality $\kappa$ such that $H$ is not a group. The example guarantees the existence of gyrogroups with cardinality $\kappa$ such that $\kappa>\omega$, which is important for the research of infinite gyrogroups.

\begin{example}\cite{BL3}
Let $\mathbb{D}$ be the gyrogroup in Example~\ref{lz1} and let $\kappa$ be an infinite cardinal number. It follows from \cite[Theorem 2.1]{ST2} that $\mathbb{D}^{\kappa}$ is a gyrogroup. Fix a subset $X$ of the gyrogroup $\mathbb{D}^{\kappa}$ such that the cardinality of $X$ is equal to $\kappa$ and $X$ contains arbitrary three points $x=(x_{\alpha})_{\alpha<\kappa}$, $y=(y_{\alpha})_{\alpha<\kappa}$ and $z=(z_{\alpha})_{\alpha<\kappa}$ of $\mathbb{D}^{\kappa}$ such that there exists $\beta<\alpha$ with $x_{\beta}=1/2, y_{\beta}=i/2$ and $z_{\beta}=-1/2$. From the proof of \cite[Example 2]{AW}, we see that $x\oplus (y\oplus z)\neq (x\oplus y)\oplus z$. Put $H=\langle X\rangle$, that is, $H$ is a subgyrogroup generated from $X$. Then the cardinality of $H$ is also equal to $\kappa$. Moreover, since $x, y, z\in H$, it follows that $H$ is not a group.
\end{example}

We recall the following concept of the coset space of a topological gyrogroup.

Let $(G, \tau, \oplus)$ be a topological gyrogroup and $H$ an $L$-subgyrogroup of $G$. It follows from \cite[Theorem 20]{ST} that $G/H=\{a\oplus H:a\in G\}$ is a coset space which defines a partition of $G$. We denote by $\pi$ the mapping $a\mapsto a\oplus H$ from $G$ onto $G/H$. Clearly, for each $a\in G$, we have $\pi^{-1}(\pi(a))=a\oplus H$. Indeed, for any $a\in G$ and $h\in H$,
\begin{eqnarray}
(a\oplus h)\oplus H&=&a\oplus (h\oplus \mbox{gyr}[h,a](H))\nonumber\\
&=&a\oplus (h\oplus \mbox{gyr}^{-1}[a,h](H))\nonumber\\
&=&a\oplus (h\oplus H)\nonumber\\
&=&a\oplus H\nonumber
\end{eqnarray}

Denote by $\tau (G)$ the topology of $G$, the quotient topology on $G/H$ is as follows: $$\tau (G/H)=\{O\subseteq G/H: \pi^{-1}(O)\in \tau (G)\}.$$

Throughout this paper, denote by $\pi$ the natural homomorphism from a topological gyrogroup $G$ to its quotient topology on $G/H$.

Then we recall some important concepts in the following researches.

\begin{definition}
Let $X$ be a topological space.

$(1)$\, $X$ is called a {\it weakly first-countable space} or {\it $gf$-countable space} \cite{Aav66} if for each point $x\in X$ it is possible to assign a sequence $\{B(n,x):n\in\mathbb{N}\}$ of subsets of $X$ containing $x$ in such a way that $B(n+1,x)\subseteq B(n,x)$ and so that a set $U$ is open if, and only if, for each $x\in U$ there exists $n\in\mathbb{N}$ such that $B(n,x)\subseteq U$.

$(2)$\, $X$ is called a {\it sequential space} \cite{FS} if for each non-closed subset $A\subseteq X$, there are a point $x\in X\setminus A$ and a sequence in $A$ converging to $x$ in $X$.

$(3)$\, $X$ is called a {\it Fr\'{e}chet-Urysohn space} \cite{FS} if for any subset $A\subseteq X$ and $x\in \overline{A}$, there is a sequence in $A$ converging to $x$ in $X$.

$(4)$\, $X$ is called a {\it strongly Fr\'{e}chet-Urysohn space} \cite{SF} if the following condition is satisfied:

(SFU) For each $x\in X$ and every sequence $\xi=\{A_{n}:n\in\mathbb{N}\}$ of subsets of $X$ such that $x\in\bigcap_{n\in\mathbb{N}}\overline{A_{n}}$, there exists a sequence $\eta=\{b_{n}:n\in\mathbb{N}\}$ in $X$ converging to $x$ and
intersecting infinitely many members of $\xi$.

$(5)$\, $X$ is called an {\it $\alpha_{4}$-space} \cite{Nt85}, if for every point $x\in X$ and each sheaf $\{S_{n}:n\in\mathbb{N}\}$ with the vertex $x$, there exists a sequence converging to $x$ which meets infinitely many sequences $S_{n}$.

$(6)$\, $X$ is called an {\it $\alpha_{7}$-space} \cite{BtZl04}, if for every point $x\in X$ and each sheaf $\{S_{n}:n\in\omega\}$ with the vertex $x$, there exists a sequence converging to some point $y\in X$ which meets infinitely many sequences $S_{n}$.
\end{definition}

\begin{definition}\cite{AA}
Let $\zeta$ be a family of non-empty subsets of a topological space $X$.

$(1)$\, $\zeta$ is called a {\it prefilter} on $X$ if whenever $P_{1}$ and $P_{2}$ are in $\zeta$, there exists $P\in\zeta$ such that $P\subseteq P_{1}\cap P_{2}$.

$(2)$\, A prefilter $\zeta$ on $X$ is said to {\it converge to a point} $x\in X$ if every open neighbourhood of $x$ contains an element
of $\zeta$.

$(3)$\, A prefilter $\zeta$ on $X$ is said to {\it accumulate to a point} $x\in X$ if $x$ belongs to the closure of each element of $\zeta$.

$(4)$\, Two prefilters $\zeta$ and $\eta$ on $X$ are said to be {\it synchronous} if, for any $P\in\eta$ and $Q\in\eta$, $P\cap Q\neq\emptyset$.

$(5)$\, $X$ is called a {\it bisequential space} if, for every prefilter $\zeta$ on $X$
accumulating to a point $x\in X$, there exists a countable prefilter $\eta$ on $X$ converging to the
same point $x$ such that $\zeta$ and $\eta$ are synchronous.
\end{definition}

The following figure lists some basic relationships of a class of weakly first countable spaces, see \cite{Lingxuewei}.

\begin{center}
\setlength{\unitlength}{1cm}
\begin{picture}(9,4)
 \thicklines
 \put(1.8,3.1){\makebox(0,0){bisequential}}
 \put(1.8,2.7){\makebox(0,0){space}}
 \put(2.9,2.9){\vector(1,0){0.5}}
 \put(5.6,3.1){\makebox(0,0){strongly Fr\'{e}chet-Urysohn}}
 \put(5.6,2.7){\makebox(0,0){space}}
 \put(7.8,2.9){\vector(1,0){0.5}}
 \put(5.6,2.6){\vector(0,-1){0.5}}
 \put(9.7,3.1){\makebox(0,0){Fr\'{e}chet-Urysohn}}
 \put(9.7,2.7){\makebox(0,0){space}}
 \put(9.7,2.6){\vector(0,-1){0.4}}
 \put(-1,2.4){\makebox(0,0){first-countable}}
 \put(-1,2.0){\makebox(0,0){space}}
 \put(0.3,2.4){\vector(1,1){0.5}}
 \put(0.3,2.4){\vector(1,-1){0.5}}
 \put(5.9,1.8){\makebox(0,0){$\alpha_{4}$-space}}
  \put(8.1,1.8){\makebox(0,0){$\alpha_{7}$-space}}
  \put(9.6,2.0){\makebox(0,0){sequential}}
  \put(9.6,1.6){\makebox(0,0){space}}
 \put(2.7,2.0){\makebox(0,0){weakly first-countable}}
 \put(2.7,1.6){\makebox(0,0){space}}
  \put(4.5,1.8){\vector(1,0){0.6}}
  \put(6.7,1.8){\vector(1,0){0.6}}
   \put(2.7,1.5){\line(0,-1){0.3}}
    \put(2.7,1.2){\line(1,0){7}}
    \put(9.7,1.2){\vector(0,1){0.3}}
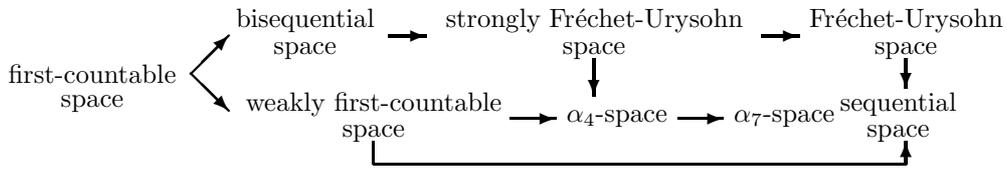
 \put(4.3,0.5){\makebox(0,0){\quad The relationships among a class of weakly first-countable spaces}}
 \end{picture}
\end{center}
\vspace{-0.2cm}

\section{Quotient spaces with inner neutral strong subgyrogroups}

In this section, it is shown that if $G$ is a strongly topological gyrogroup, $H$ is a closed strong subgyrogroup of $G$ and $H$ is inner neutral, then the quotient space $G/H$ is first-countable if and only if $G/H$ is a bisequential space if and only if $G/H$ is a weakly first-countable space if and only if $G/H$ is a $csf$-countable and sequential $\alpha_{7}$-space.

\begin{definition}\cite{BX2022}
A subgyrogroup $H$ of a topological gyrogroup $G$ is called {\it strong subgyrogroup} if for any $x,y\in G$, we have $gyr[x,y](H)=H$.
\end{definition}

\begin{definition}\cite{BL1}
A subgyrogroup $H$ of a topological gyrogroup $G$ is called {\it admissible} if there exists a sequence $\{U_{n}:n\in \mathbb{N}\}$ of open symmetric neighborhoods of the identity $0$ in $G$ such that $U_{n+1}\oplus (U_{n+1}\oplus U_{n+1})\subseteq U_{n}$ for each $n\in \mathbb{N}$ and $H=\bigcap _{n\in \mathbb{N}}U_{n}$. If $G$ is a strongly topological gyrogroup with a symmetric neighborhood base $\mathscr U$ at $0$ and each $U_{n}\in \mathscr U$, we say that the admissible topological subgyrogroup is generated from $\mathscr U$.
\end{definition}

Obviously, every strong subgyrogroup is an $L$-subgyrogroup. Moreover, in a strongly topological gyrogroup with neighborhood base $\mathscr U$ of $0$, it is not difficult to see that each admissible subgyrogroup generated from $\mathscr U$ is a strong subgyrogroup. Moreover, the authors claimed that every strongly topological gyrogroup $G$ contains some strong subgyrogroups which are union-generated from open neighborhoods of the identity element by construction, see \cite[Proposition 3.11]{BX2022}.

\begin{lemma}\label{t00000}\cite{BL}
Let $G$ be a topological gyrogroup and $H$ an $L$-subgyrogroup of $G$. Then the natural homomorphism $\pi$ from a topological gyrogroup $G$ to its quotient topology on $G/H$ is an open and continuous mapping.
\end{lemma}

\begin{lemma}\label{homogeneous}\cite{BX2022}
Let $G$ be a strongly topological gyrogroup and $H$ a closed strong subgyrogroup of $G$. Then the family $\{\pi (x\oplus V):V\in \tau ,0\in U\}$ is a local base of the space $G/H$ at the point $x\oplus H\in G/H$, and $G/H$ is a homogeneous $T_{1}$-space.
\end{lemma}

A topological space $X$ is called a coset space if $X$ is homeomorphic to $G/H$, for some closed subgroup $H$ of a topological group $G$. It is well-known that every first-countable topological group is metrizable
by the Birkhoff-Kakutani theorem. However, the following example shows that it does not hold in coset spaces.

\begin{example}\label{countexample}\cite{FmSiTm19}
The two arrows space is a compact coset space which is first-countable, but not submetrizable.
\end{example}

Since every topological group is a strongly topological gyrogroup and each subgroup is a strong subgyrogroup, the Example \ref{countexample} shows that the axioms of first-countability is not equivalent with metrizability in the quotient space $G/H$, where $G$ is a strongly topological gyrogroup and $H$ is a closed strong subgyrogroup of $G$. However, if $G$ is a strongly topological gyrogroup with neighborhood base $\mathscr U$ of $0$ and $P$ is an admissible subgyrogroup generated from $\mathscr U$, if the quotient space $G/P$ is first-countable, then it is metrizable.

Indeed, in \cite{BL1}, Bao and Lin proved that if $G$ is a strongly topological gyrogroup with the symmetric neighborhood base $\mathscr{U}$ at $0$ and $P$ is an admissible $L$-subgyrogroup of $G$ generated from $\mathscr U$, then the left coset space $G/P$ is submetrizable, which means that the topology generated by a metric on $G/P$ is coarser than the quotient topology. Therefore, it suffices to show that if the the quotient space $G/P$ is first-countable, then the quotient topology is coarser than the topology generated by the same metric on $G/P$. However, for convenience of reading, we write the main process of proof about this result below.

\begin{theorem}\label{metri}
Let $G$ be a strongly topological gyrogroup with a symmetric neighborhood base $\mathscr U$ at $0$ and $P$ an admissible subgyrogroup generated from $\mathscr U$. If the quotient space $G/P$ is first-countable, then it is metrizable.
\end{theorem}

\begin{proof}
Let $\{U_{n}:n\in \mathbb{N}\}$ be a sequence of symmetric open neighborhoods of the identity $0$ in $G$ satisfying $U_{n}\in \mathscr U$ and $U_{n+1}\oplus (U_{n+1}\oplus U_{n+1})\subseteq U_{n}$, for each $n\in \mathbb{N}$, and such that $P=\bigcap _{n\in \mathbb{N}}U_{n}$. By \cite[Lemma 3.12]{BL}, there exists a continuous prenorm $N$ on $G$ which satisfies $$N(\mbox{gyr}[x,y](z))=N(z)$$ for any $x, y, z\in G$ and $$\{x\in G: N(x)<1/2^{n}\}\subseteq U_{n}\subseteq\{x\in G: N(x)\leq 2/2^{n}\},$$ for any $n\in \mathbb{N}$.

It is easy to show that $N(x)=0$ if and only if $x\in P$ and $N(x\oplus p)=N(x)$ for every $x\in G$ and $p\in P$. Then define a function $d$ from $G\times G$ to $\mathbb{R}$ by $d(x,y)=|N(x)-N(y)|$ for all $x,y\in G$. It is obvious that $d$ is continuous and $d$ is a pseudometric. Therefore, if $x'\in x\oplus P$ and $y'\in y\oplus P$, then there exist $p_{1},p_{2}\in P$ such that $x'=x\oplus p_{1}$ and $y'=y\oplus p_{2}$, then $$d(x', y')=|N(x\oplus h_{1})-N(y\oplus h_{2})|=|N(x)-N(y)|=d(x, y).$$ This enables us to define a function $\varrho $ on $G/P\times G/P$ by $$\varrho (\pi _{P}(x),\pi _{P}(y))=d(\ominus x\oplus y, 0)+d(\ominus y\oplus x, 0)$$ for any $x, y\in G$.

It is obvious that $\varrho $ is continuous, and $\varrho $ is a metric on $G/H$.

Finally, it suffices to verify that $\varrho$ generates the quotient topology of the space $G/H$. Given any points $x\in G$, $y\in G/H$ and any $\varepsilon>0$, we define open balls, $$B(x, \varepsilon)=\{x'\in G: d(x',x)<\varepsilon\}$$ and $$B^{*}(y, \varepsilon)=\{y'\in G/H: \varrho (y',y)<\varepsilon\}$$ in $G$ and $G/H$, respectively. Obviously, if $x\in G$ and $y=\pi _{P}(x)$, then we have $B(x, \varepsilon)=\pi ^{-1}_{P}(B^{*}(y, \varepsilon))$. Therefore, the topology generated by $\varrho$ on $G/P$ is coarser than the quotient topology.

Moreover, it follows from Lemma \ref{homogeneous} that the quotient space $G/P$ is homogenous. Since $G/P$ is first-countable, there exists a countable base $\mathcal{V}=\{V_{n}:n\in \mathbb{N}\}$ at $\pi (0)$ in $G/P$. As $\pi$ is an open and continuous mapping from $G$ onto $G/P$ by Lemma \ref{t00000}, $\pi(W)$ is an open neighborhood of $\pi(0)$ for arbitrary open neighborhood $W$ of $\pi^{-1}(\pi(0))$. We can find $V_{m}\in \mathcal{V}$ such that $\pi(0)\subseteq V_{m}\subseteq \pi(W)$. Then $\pi^{-1}(\pi(0))\subseteq \pi^{-1}(V_{m})\subseteq W$. Then $P$ has a countable character in $G$. Suppose that the preimage $O=\pi ^{-1}_{P}(Q)$ is open in $G$, where $Q$ is a non-empty subset of $G/P$. For every $y\in Q$, we have $\pi ^{-1}_{P}(y)=x\oplus P\subseteq O$, where $x$ is an arbitrary point of the fiber $\pi ^{-1}_{P}(y)$. Since $\{\pi^{-1}(V_{n}): n\in \omega\}$ is a base for $G$ at $P$, there exists $n\in \omega$ such that $x\oplus \pi^{-1}(V_{n})\subseteq O$. Then there exists $\delta>0$ such that $B(x, \delta)\subseteq x\oplus \pi^{-1}(V_{n})$. Therefore, we have $\pi ^{-1}_{P}(B^{*}(y, \delta))=B(x, \delta)\subseteq x\oplus \pi^{-1}(V_{n})\subseteq O$. It follows that $B^{*}(y, \delta)\subseteq Q$. So the set $Q$ is the union of a family of open balls in $(G/P, \varrho)$. Hence, $Q$ is open in $(G/P, \varrho)$, which proves that the metric and quotient topologies on $G/P$ coincide.
\end{proof}

By Example \ref{countexample} and Theorem \ref{metri}, we know that in a strongly topological gyrogroup $G$ with a symmetric neighborhood base $\mathscr U$ at $0$, even though the admissible subgyrogroup $P$ generated from $\mathscr U$ is a strong subgyrogroup and both of them can be a coset to define a partition of $G$, respectively, but the quotient spaces $G/P$ and $G/H$ have some different properties.

\begin{definition}\cite{BX2022}
A subgyrogroup $H$ of a topological gyrogroup $G$ is called {\it inner (outer) neutral} if for every open neighborhood $U$ of $0$ in $G$, there exists an open neighborhood $V$ of $0$ such that $H\oplus V\subseteq U\oplus H$ ($V\oplus H\subseteq H\oplus U$).
\end{definition}

It was proved in \cite{BX2022} that if $G$ is a strongly topological gyrogroup, $H$ is a closed strong subgyrogroup of $G$ and $H$ is inner neutral, then $G/H$ is first-countable if and only if $G/H$ is Fr\'echet-Urysohn with an $\omega^{\omega}$-base. Here, we continue to study some properties about the axioms of first-countability in the quotient spaces $G/H$, where $G$ is a strongly topological gyrogroup and $H$ is a inner neutral closed strong subgyrogroup of $G$.

\begin{theorem}
Suppose that $G$ is a strongly topological gyrogroup, $H$ is a closed strong subgyrogroup of $G$ and $H$ is inner neutral, then the followings are equivalent.

$(1)$\, $G/H$ is a sequential $\alpha_{4}$-space;

$(2)$\, $G/H$ is Fr\'{e}chet-Urysohn;

$(3)$\, $G/H$ is strongly Fr\'{e}chet-Urysohn.
\end{theorem}

\begin{proof}
It suffices to prove that $(1)\Rightarrow (2)\Rightarrow (3)$. Suppose further that the space $G/H$ is non-discrete.

$(1)\Rightarrow (2)$.\, For  $A\subseteq G/H$, we write $[A]$  the set of all limit points of sequences in $A$. Suppose on the contrary that $G/H$ is not Fr\'{e}chet-Urysohn. There is a subset $B$ of $G/H$ such that $[B]\neq \overline{B}$. If $[B]$ is closed in $G/H$, then $\overline{B}\subseteq \overline{[B]}=[B]\subseteq\overline{B}$, which is a contradiction. Hence, $[B]$ is not closed in $G$. By the hypothesis, $G/H$ is sequential, so $[B]$ is not sequentially closed, that is $[[B]]\neq[B]$. Thus there is $b\in[[B]]\setminus[B]$. We may assume $b=\pi(0)$ without loss of generality, since $G/H$ is homogeneous by Lemma \ref{homogeneous}.

Let $\{b_{n}:n\in\mathbb{N}\}$ be a sequence of points of $[B]$ converging to $\pi(0)$. For each $n\in\mathbb{N}$, fix a point $x_{n}\in \pi^{-1}(b_{{n}})$. For each $b_{{n}}$, let $\{b_{{n}}(j):j\in\mathbb{N}\}$ be a sequence of points of $B$ converging to $b_{{n}}$. For each $j\in\mathbb{N}$, fix a point $x_{n}(j)\in \pi^{-1}(b_{{n}}(j))$. We claim $\lim_{j\to\infty}\pi((\ominus x_{n})\oplus x_{n}(j))=\pi(0)$.

Indeed, let $O$ be an open neighborhood of $\pi(0)$ in $G/H$, then there is an open neighborhood $U$ of $0$ in $G$ such that $\pi(U)\subseteq O$. Since $\lim_{j\to\infty}\pi(x_{n}(j))=\lim_{j\to\infty}b_{n}(j)=b_{n}=\pi(x_{n})$, there is $m\in\mathbb{N}$ such that $\pi(x_{n}(j))\in \pi(x_{n}\oplus U)$ for $j\geq m$. So $\pi((\ominus x_{n})\oplus x_{n}(j))\in \pi((\ominus x_{n})\oplus (x_{n}\oplus U))=\pi(U)\subseteq O$ for $j\geq m$. Hence $\lim_{j\to\infty}\pi((\ominus x_{n})\oplus x_{n}(j))=\pi(0)$.

Since $G/H$ is an $\alpha_{4}$-space, it is possible to pick $n_{k},j_{k}$ for each $k\in\mathbb{N}$ such that $\{\pi ((\ominus x_{n_{k}})\oplus x_{n_{k}}(j_{k})): k\in\mathbb{N}\}$ converges to $\pi (0)$ and $n_{k}<n_{k+1}$ for each $k\in\mathbb{N}$. It follows from \cite[Lemma 3.17]{BX2022} that $\lim_{k\to\infty}b_{n_{k}}(j_{k})=\lim_{k\to\infty}\pi (x_{n_{k}}(j_{k}))=\lim_{k\to\infty}\pi (x_{n_{k}}\oplus ((\ominus x_{n_{k}})\oplus x_{n_{k}}(j_{k})))=\pi (0)$,
this contradicts the assumption that $\pi (0)\not\in[B]$. Therefore, $G/H$ is Fr\'{e}chet-Urysohn.

$(2)\Rightarrow (3)$.\, It is enough to verify condition (SFU) for $\pi (0)$. Suppose that $\pi (0)\in\bigcap_{n\in\mathbb{N}}\overline{B_{n}}$, where each $B_{n}$ is a subset of $G/H$. Fix a sequence $\{x_{n}:n\in\mathbb{N}\}$ in $G/H\setminus\{\pi (0)\}$ converging to $\pi (0)$. For each $n\in\mathbb{N}$, fix a point $a_{n}\in \pi^{-1}(x_{n})$. Since $G$ is a strong topological gyrogroup and $H$ is inner neutral, we can fix symmetric open neighbourhoods $U_{n}$, $V_{n}$ of $0$ such that $\pi (a_{n})\not\in \pi (U_{n}\oplus U_{n})$, $V_{n}\subseteq U_{n}$ and $H\oplus V_{n}\subseteq U_{n}\oplus H$ for each $n\in\mathbb{N}$. Let $A_{n}=\pi^{-1}(B_{n})$. Since $\pi$ is open continuous by Lemma \ref{t00000}, $0\in \bigcap_{n\in\mathbb{N}}\pi^{-1}(\overline{B_{n}})=\bigcap_{n\in\mathbb{N}}\overline{\pi^{-1}({B_{n}})}=\bigcap_{n\in\mathbb{N}}\overline{A_{n}}$. Since $0\in \overline{A_{n}}$, we may assume that $A_{n}\subseteq V_{n}$, for each $n\in\mathbb{N}$ (otherwise, replace $A_{n}$ with the intersection $A_{n}\cap V_{n}$). Put $C_{n}=a_{n}\oplus A_{n}$, for $n\in\mathbb{N}$. From the choice of $V_{n}$ it is clear that $\pi (0)\not\in \overline{\pi (C_{n})}$, while $x_{{n}}\in \overline{\pi (C_{n})}$, for $n\in\mathbb{N}$. The last condition and the fact that $\{x_{{n}}:n\in\mathbb{N}\}$ converges to $\pi (0)$ implies that $\pi (0)\in \overline{\pi (C)}$, where $C=\bigcup\{C_{n}:n\in\mathbb{N}\}$.

Since the space $G/H$ is Fr\'{e}chet-Urysohn, there exists a sequence $\eta=\{d_{n}:n\in\mathbb{N}\}$ in $\pi (C)$ converging to $\pi (0)$. Since $\pi (0)$ is not in the closure of $\pi (C_{n})$, the sequence $\eta$ must intersect $\pi (C_{n})$ for infinitely many of $n$. For every $n\in\mathbb{N}$, choose $k_{n}\in\mathbb{N}$ such that $d_{n}\in \pi (C_{k_{n}})$. Hence, there exists $c_{n}\in C_{k_{n}}$ such that $\pi (c_{n})=d_{n}$ for each $n\in\mathbb{N}$. Put $b_{n}=(\ominus (a_{k_{n}}))\oplus c_{n}$, then  $\pi (b_{n})\in B_{k_{n}}$. It follows from \cite[Lemma 3.17]{BX2022} that $\lim_{n\to\infty}\pi (b_{n})=\pi (0)$. Thus, condition (SFU) is satisfied, and the space $G/H$ is strongly Fr\'{e}chet-Urysohn.
\end{proof}

\begin{lemma}\label{3yl1}
Let $G$ be a strongly topological gyrogroup and $H$ a closed strong subgyrogroup of $G$. If $U$ and $V$ are open neighborhoods of $0$ in $G$ with $(\ominus V)\oplus V\subseteq U$. Then $\overline{\pi (V)}\subseteq \pi (U)$.
\end{lemma}

\begin{proof}
Take any $x\in G$ such that $\pi(x)\in \overline{\pi(V)}$. Since $V\oplus x$ is an open neighborhood of $x$, $\pi(V\oplus x)$ is an open neighborhood of $\pi(x)$ by Lemma \ref{t00000}. Then $\pi(V\oplus x)\cap \pi(V)\not=\emptyset$. We can find $a\in V$ and $b\in V$ such that $\pi (a\oplus x)=\pi(b)$, that is, $a\oplus x=b\oplus h$, for some $h\in H$. Then
\begin{eqnarray}
x&=&(\ominus a)\oplus (b\oplus h)\nonumber\\
&=&((\ominus a)\oplus b)\oplus \mbox{gyr}[\ominus a,b](h)\nonumber\\
&\in&((\ominus a)\oplus b)\oplus H\nonumber\\
&\subseteq& ((\ominus V)\oplus V)\oplus H\nonumber\\
&\subseteq& U\oplus H\nonumber.
\end{eqnarray}
Thus, $\pi (x)\in \pi (U\oplus H)=\pi (U)$, which means that $\overline{\pi (V)}\subseteq \pi (U)$.
\end{proof}

\begin{theorem}\label{regular}
Let $G$ be a strongly topological gyrogroup and $H$ a closed strong subgyrogroup of $G$. Then the quotient space $G/H$ is regular.
\end{theorem}

\begin{proof}
Let $W$ be an arbitrary open neighborhood of $\pi (0)$ in $G/H$. By Lemma \ref{t00000}, the natural quotient mapping $\pi$ is open and continuous, so there exists an open neighborhood $U$ of $0$ in $G$ such that $\pi (U)\subseteq W$. It follows from \cite[Proposition 8]{AW} that we can find an open neighborhood $V$ of $0$ such that $(\ominus V)\oplus V\subseteq U$. By Lemma \ref{3yl1}, $\overline{\pi(V)}\subseteq \pi (U)\subseteq W$. It is clear that $\pi (V)$ is an open neighborhood of $\pi (0)$ in $G/H$, so we obtain that $G/H$ is regular at the point $\pi (0)$. Moreover, it is verified that $G/H$ is homogeneous by Lemma \ref{homogeneous}, hence the quotient space $G/H$ is regular.
\end{proof}

\begin{theorem}
Suppose that $G$ is a strongly topological gyrogroup with neighborhood base $\mathscr U$ of $0$, $H$ is a closed strong subgyrogroup of $G$ and $H$ is inner neutral, if the space $G/H$ is bisequential, then $G/H$ is first-countable.
\end{theorem}

\begin{proof}
Since the space $G/H$ is regular by Theorem \ref{regular} and is also bisequential, we can find a countable open prefilter $\xi$ on $G/H$ converging to $\pi (0)$ by \cite[Lemma 4.7.11]{AA}. Let $Q_{P}=\pi^{-1}(P)$ for each $P\in\xi$. Put $\gamma=\{\pi((\ominus Q_{P})\oplus Q_{P}):P\in\xi\}$. Then $\gamma$ is a base at $\pi(0)$ in $G/H$. Indeed, all elements of $\gamma$ are open in $G/H$ and contain $\pi(0)$. Let $O$ be an open neighborhood of $\pi(0)$. Take open symmetric neighborhoods $U,V\in \mathscr U$ in $G$ such that $V\oplus V\subseteq U$, $H\oplus (V\oplus V)\subseteq U\oplus H$ and $\pi(U)\subseteq O$. Since $\xi$ converges to $\pi(0)$, there exists $P\in\xi$ such that $P\subseteq\pi(V)$. It follows that $Q_{P}\subseteq V\oplus H$. Then
\begin{eqnarray}
\ominus Q_{P}&\subseteq &\ominus (V\oplus H)\nonumber\\
&=&\bigcup_{v\in V,h\in H}\{\ominus (v\oplus h)\}\nonumber\\
&=&\bigcup_{v\in V,h\in H}\{gyr[v,h](\ominus h\ominus v)\}\nonumber\\
&=&\bigcup_{v\in V,h\in H}\{gyr[v,h](\ominus h)\oplus gyr[v,h](\ominus v)\}\nonumber\\
&=&H\oplus V.\nonumber
\end{eqnarray}
Therefore, $0\in (\ominus Q_{P})\oplus Q_{P}\subseteq (H\oplus V)\oplus (V\oplus H)$ and
\begin{eqnarray}
(H\oplus V)\oplus (V\oplus H)&=&\bigcup_{v_{1},v_{2}\in V,h_{1},h_{2}\in H}\{(h_{1}\oplus v_{1})\oplus (v_{2}\oplus h_{2})\}\nonumber\\
&=&\bigcup_{v_{1},v_{2}\in V,h_{1},h_{2}\in H}\{((h_{1}\oplus v_{1})\oplus v_{2})\oplus gyr[h_{1}\oplus v_{1},v_{2}](h_{2})\}\nonumber\\
&=&\bigcup_{v_{1},v_{2}\in V,h_{1},h_{2}\in H}\{(h_{1}\oplus (v_{1}\oplus gyr[v_{1},h_{1}](v_{2})))\oplus gyr[h_{1}\oplus v_{1},v_{2}](h_{2})\}\nonumber\\
&=&(H\oplus (V\oplus V))\oplus H\nonumber\\
&\subseteq&(U\oplus H)\oplus H,\nonumber
\end{eqnarray}
hence, $\pi(0)\in\pi((\ominus Q_{P})\oplus Q_{P})\subseteq\pi((U\oplus H)\oplus H)=\pi (U)\subseteq O$. Hence, $G/H$ is first-countable.

\end{proof}

It follows from \cite[Corollary 1.3.10(2)]{linbook1} that every weakly first-countable Fr\'{e}chet-Urysohn space is first-countable, so it is clear that if $G$ is a strongly topological gyrogroup, $H$ is a inner neutral and closed strong subgyrogroup of $G$ and the space $G/H$ is weakly first-countable, then $G/H$ is first-countable. Next, we give the other type of methods to that the equivalence between the properties of weakly first-countable and first-countable in $G/H$.

\begin{definition}
Let $\mathscr{P}$ be a family of subsets of a space $X$ with $x\in\bigcap\mathscr{P}$.

$(1)$\, The family $\mathscr{P}$ is called a {\it network at $x$} \cite{E} if for each neighborhood $U$ of $x$ there exists $P\in\mathscr{P}$ such that $P\subseteq U$.

$(2)$\, The family $\mathscr{P}$ is called a {\it $cs$-network at $x$} \cite{Ls97} if for any sequence $L$ converging to $x$ and a neighborhood $U$ of $x$, there exists $P\in\mathscr{P}$ such that $L$ is eventually in $P$ and $P\subseteq U$.

$(3)$\, The family $\mathscr{P}$ is called an {\it $sn$-network at $x$} \cite{Ls96} if $\mathscr{P}$ is a network at $x$ and each element of $\mathscr{P}$ is a sequential neighborhood of $x$.

$(4)$\, The family $\mathscr{P}$ is called an {\it $so$-network at $x$} \cite{Ls96} if $\mathscr{P}$ is a network at $x$ and each element of $\mathscr{P}$ is a sequential open subset of $X$.

$(5)$\, A space $X$ is called {\it $csf$-countable} (resp., {\it $snf$-countable}, {\it $sof$-countable}) \cite{Ls96} if for each $x\in X$, there is a countable $cs$-network (resp., $sn$-network, $so$-network) at $x$.
\end{definition}

Note that in \cite{BtZl04}, $csf$-countable spaces and $snf$-countable spaces are called {\it spaces with countable $cs^{\ast}$-character} and {\it spaces with countable $sb$-character}, respectively. A family $\mathscr{P}$ of subsets of a space $X$ is called a {\it $cs$-network} if for each $x\in X$, $\{P\in\mathscr{P}:x\in P\}$ is a $cs$-network at $x$.

According to the process of the proof of \cite[Lemma 4.7.1]{AA} and \cite[Proposition 4.7.2]{AA}, we can easily obtain the following two lemmas:

\begin{lemma}\label{l31'}
Let $\{V_{n}(x):n\in\omega,x\in X\}$ and $\{W_{n}(x):n\in\omega,x\in X\}$ be two $sn$-networks on a Hausdorff space $X$. Then, for each $x\in X$ and each $n\in\omega$, there is $m\in\omega$ such that $W_{m}(x)\subseteq V_{n}(x)$.
\end{lemma}

\begin{lemma}\label{l53}
Let $\{V_{n}(x):n\in\omega,x\in X\}$ be an $sn$-network on a homogeneous Hausdorff space $X$, and let $b$ be an element of $X$. Suppose further that, for each $x\in X$, $f_{x}$ is a homeomorphism of $X$ onto $X$ such that $f_{x}(b)=x$. Put $W_{n}(x)=f_{x}(V_{n}(b))$. Then $\{W_{n}(x):n\in\omega,x\in X\}$ is an $sn$-network on $X$.
\end{lemma}

\begin{theorem}
Suppose that $G$ is a strongly topological gyrogroup with neighborhood base $\mathscr U$ of $0$, $H$ is a closed strong subgyrogroup of $G$ and $H$ is inner neutral, if the space $G/H$ is an $snf$-countable space, then $G/H$ is an $sof$-countable space.
\end{theorem}
\begin{proof}
For each $x\in G/H$, fix a point $a_{x}\in G$ such that $\pi (a_{x})=x$, where $a_{\pi (0)}=0$. For each subset $K$ of $G/H$, put $T_{K}=\pi^{-1}(K)$.

Let $\{O_{n}(x):n\in\omega,x\in G/H\}$ be an $sn$-network on $G/H$. Put $Q_{n}(x)=\pi (a_{x}\oplus (T_{O_{n}(\pi (0))}\oplus T_{O_{n}(\pi (0))}))$ for each $x\in G/H$ and $n\in\omega$.

\medskip
{\bf Claim 1:} $\{Q_{n}(x):n\in\omega,x\in G/H\}$ is a network on $G/H$.

We can assume that $O_{n}(x)=\pi (a_{x}\oplus T_{O_{n}(\pi (0))})$ for each $x\in G/H$ and $n\in\omega$. Indeed, this follows from Lemma \ref{l53} since $h_{a}$ on $G/H$ given by the formula $h_{a}(b\oplus H)=(a\oplus b)\oplus H$, for each $a\in G$, is a homeomorphism of $G/H$ onto itself. Next, we only need prove that $\{Q_{n}(x):n\in\omega,x\in G/H\}$ is a network of $G/H$. Indeed, let $V$ be a neighborhood of a point $x\in G/H$. Since $G$ is a strongly topological gyrogroup and $H$ is inner neutral, there exists $W\in\mathscr U$ such that $\pi (a_{x}\oplus ((W\oplus H)\oplus (W\oplus H)))\subseteq V$. It follows from $\pi (0)\in \pi (W\oplus H)$ that there exists $n\in\omega$ such that $O_{n}(\pi (0))\subseteq \pi (W\oplus H)$. So $x\in Q_{n}(x)=\pi (a_{x}\oplus (T_{O_{n}(\pi (0))}\oplus T_{O_{n}(\pi (0))}))\subseteq V$, whence $\{Q_{n}(x):n\in\omega,x\in G/H\}$ is a network of $G/H$.

\medskip
{\bf Claim 2:} $O_{n}(\pi (0))$ contains a sequentially open neighbourhood of $\pi (0)$ for each $n\in\omega$.

Let $B_{n}$ be the set of all points $x\in O_{n}(\pi (0))$ such that $\pi (a_{x}\oplus T_{O_{k}(\pi (0))})\subseteq O_{n}(\pi (0))$, for some $k\in\omega$. Clearly, $\pi (0)\in B_{n}\subseteq O_{n}(\pi (0))$. We claim that the set $B_{n}$ is sequentially open in $G/H$. Indeed, take any $y\in B_{n}$. Then $\pi (a_{y}\oplus T_{O_{k}(\pi (0))})\subseteq O_{n}(\pi (0))$, for some $k\in\omega$. By Claim 1 and Lemma \ref{l31'}, there is $m\in\omega$ such that $\pi (T_{O_{m}(\pi (0))}\oplus T_{O_{m}(\pi (0))})\subseteq O_{k}(\pi (0))$. Then $\pi (a_{y}\oplus (T_{O_{m}(\pi (0))}\oplus T_{O_{m}(\pi (0))}))\subseteq \pi (a_{y}\oplus T_{O_{k}(\pi (0))})\subseteq O_{n}(\pi (0))$, which implies that $O_{m}(y)=\pi (a_{y}\oplus T_{O_{m}(\pi (0))})\subseteq B_{n}$. Since $y$ was an arbitrary point of $B_{n}$, it follows that $B_{n}$ is sequentially open in $G/H$.

Now it is clear that $\{B_{n}:n\in\omega\}$ is a countable $so$-network of $G/H$ at $\pi (0)$. Hence $G/H$ is an $sof$-countable space.
\end{proof}

\begin{corollary}
Suppose that $G$ is a strongly topological gyrogroup, $H$ is a closed strong subgyrogroup of $G$ and $H$ is inner neutral, if the space $G/H$ is weakly first-countable, then $G/H$ is a first-countable space.
\end{corollary}

A partially ordered set $(T,\leq)$ is called a {\it tree} if, for each $t\in T$, the set $\mathord{\downarrow} t=\{s\in T:s\leq t\}$ is well-ordered by the order $\leq$. Given an element $t\in T$, let $\mathord{\uparrow} t=\{s\in T:s\geq t\}$ and succ$(t)=\mbox{min}(\mathord{\uparrow} t\setminus\{t\})$ be the set of successors of $t$ in $T$. A maximal linearly ordered subset of a tree $T$ is called a {\it branch} of $T$. By max$(T)$ (min$(T)$) we denote the set of maximal (minimal) elements of a tree $T$. Under a {\it sequence tree} in a space $X$ we understand a tree $(T,\leq)$ such that (1) $T\subseteq X$; (2) $T$ has no infinite branch; (3) for each $t\not\in\mbox{max}(T)$, the set succ$(t)$ is countable and converges to $t$.

\begin{theorem}
Let $G$ be a strongly topological gyrogroup with neighborhood base $\mathscr U$ of $0$, $H$ a inner neutral and closed strong subgyrogroup of $G$. If the space $G/H$ is a $csf$-countable and sequential $\alpha_{7}$-space, then $G/H$ is first-countable.
\end{theorem}

\begin{proof}
Let $\mathcal{A}$ be a countable $cs$-network at $\pi (0)$ in $G/H$. Without loss of generality, we may assume that the family $\mathcal{A}$ is closed under finite unions and finite intersections. For each $x\in G/H$, fix $a_{x}\in G$ such that $\pi (a_{x})=x$, where $a_{\pi (0)}=0$. For each subset $K$ of $G/H$, put $T_{K}=\pi^{-1}(K)$. Since $G$ is a strongly topological gyrogroup and $H$ is inner neutral in $G$, we can assume that $\pi (T_{A_{1}}\oplus (T_{A_{2}}\oplus (\cdots \oplus (T_{A_{n-1}}\oplus T_{A_{n}}))\cdots )\in\mathcal{A}$ if $A_{i}\in\mathcal{A}$ for each $i\leq n$. It suffices to prove that the subcollection $\mathcal{A}'=\{A\in\mathcal{A}:A\mbox{~is a sequential neighborhood of~}\pi (0)\}$ of $\mathcal{A}$ is a network at $\pi (0)$ in $G/H$, which implies that $G/H$ has a countable $sn$-network at $\pi (0)$.

Assuming the contrary, then there exists $U\in\mathscr U$ such that $A\not\subseteq p(U)$ for any $A\in\mathcal{A}'$. Let $\mathcal{A}''=\{A\in\mathcal{A}:A\subseteq p(U)\}$. Then $\mathcal{A}'\cap\mathcal{A}''=\emptyset$. Since $\mathcal{A}''$ is a countable family, we may write $\mathcal{A}''=\{A_{n}:n\in\omega\}$. Put $B_{n}=\bigcup_{k\leq n}A_{k}$ for every $n\in\omega$. Note that $B_{n}\in\mathcal{A}''$ for every $n\in\omega$.

Let $m_{-1}=0$ and $U_{-1}=U$. Since $B_{0}$ is not a sequential neighborhood at $\pi (0)$ in $G/H$, there exists a sequence $L_{0}=\{x_{0,i}\}_{i\in\omega}$ in $G/H$ such that $L_{0}$ converging to $\pi (0)$ and $L_{0}\cap B_{0}=\emptyset$. Since the space $G/H$ is regular by Theorem \ref{regular}, we can choose $U_{0}\in\mathscr U$ such that $(U_{0}\oplus H)\oplus (U_{0}\oplus H)\subseteq U\oplus H$ and $\overline{\pi ((U_{0}\oplus H)\oplus (U_{0}\oplus H))}\subseteq \pi (U)$. It follows that there exists $m_{0}\in\omega$ such that $L_{0}$ is eventually in $A_{m_{0}}\subseteq \pi (U_{0})$.  Without loss of generality, we can assume that $L_{0}\subseteq A_{m_{0}}$. By induction, we can construct $L_{k}$,  $m_{k}$ and $U_{k}\in\mathscr U$ such that for each $k\in\omega$,

(\expandafter{\romannumeral1}) $m_{k}>m_{k-1}$, $L_{k}=\{x_{k,i}\}_{i\in\omega}$ converges to $\pi (0)$ and $L_{k}\subseteq p(U_{k})\setminus B_{m_{k-1}}$;

(\expandafter{\romannumeral2}) $L_{k}\subseteq A_{m_{k}}\subseteq B_{m_{k}}$;

(\expandafter{\romannumeral3}) $\overline{\pi (U_{k})}\cap\{x_{j,i}:j,i<k\}=\emptyset$ and $\overline{\pi ((U_{k}\oplus H)\oplus (U_{k}\oplus H))}\subseteq \pi (U_{k-1})$.

Let $X=\bigcup_{k\in\omega}L_{k}$, $Y=\overline{X}\setminus X$. Then $X$ is a discrete subspace of $G/H$. It follows that $Y$ is closed in $G/H$. Consider the following two cases.

{\bf Case 1.} $\pi (0)$ is an isolated point of $Y$.

We can find $W\in\mathscr U$ such that $\overline{p(W)}\cap Y=\{\pi (0)\}$. It follows that $X'=(X\cup\{\pi (0)\})\cap \overline{p(W)}=(X\cup Y)\cap \overline{p(W)}=\overline{X}\cap \overline{p(W)}$ is closed in $G/H$. For every $k\in\omega$ consider the sequence $S_{k}=L_{k}\cap \overline{p(W)}$ convergent
to $\pi (0)$. Since $G/H$ is an $\alpha_{7}$-space, there is a convergent sequence $S\subseteq G/H$ such that $S\cap S_{k}\neq\emptyset$ for infinitely many sequences $S_{k}$. Taking into account that $X'$ is a closed subspace of $G/H$ with $|X'\cap S|=\aleph_{0}$, we conclude that the limit point lim$S$ of $S$ belongs to the set $X'$. Moreover, we can assume that $S\subseteq X'$. Since the space $X$ is discrete, lim$S\in X'\setminus X=\{\pi (0)\}$. Thus the sequence $S$ converges to $\pi (0)$. Since $\mathcal{A}''$ is a $cs$-network at $\pi (0)$ in $G/H$, there is a number $k_{0}\in\omega$ such that $S$ is eventually in $B_{m_{k_{0}}}$. Then $S$ cannot meet infinitely many sequences $S_{k}$, this contradicts to the choice of $S$.

{\bf Case 2.} $\pi (0)$ is a cluster point of $Y$.

Since $Y$, as a closed subspace of $G/H$, is sequential, we may pick a non-trival sequence in $Y$ converging to $\pi (0)$. Then it follows from \cite[Lemma 1]{BtZl04} that there exists a sequential tree $T\subseteq \overline{X}$ such that min$(T)=\{\pi (0)\}$, max$(T)\subset X$ and succ$(\pi (0))\subseteq Y$. Let $t_{0}=\pi (0)$. Since $\mathcal{A}$ is a $cs$-network at $\pi (0)$ in $G/H$, there is $C_{0}\in\mathcal{A}$ such that $C_{0}\subseteq \pi (U_{0})$ and succ$(\pi (0))$ is eventually in $C_{0}$. Pick $t_{1}\in \mbox{succ}(t_{0})\cap C_{0}$. By induction, we can construct a finite branch $\{t_{i}:i\leq n+1\}$ of the tree $T$ and a sequence $\{C_{i}:i\leq n\}$ of elements of $\mathcal{A}$ such that succ$(t_{i})$ is eventually in $\pi (a_{t_{i}}\oplus T_{C_{i}})$, $C_{i}\subseteq \pi (U_{i})$ and $t_{i+1}\in \mbox{succ}(t_{i})\cap \pi (a_{t_{i}}\oplus T_{C_{i}})$ for each $i\leq n$. Note that the infinite set $M=\mbox{succ}(t_{n})\cap \pi (a_{t_{n}}\oplus T_{C_{n}})$ converges to the point $t_{n}\neq \pi (0)$. In addition, $M\subseteq \pi (a_{t_{n}}\oplus T_{C_{n}})\subseteq \pi (a_{t_{n-1}}\oplus (T_{C_{n-1}}\oplus T_{C_{n}}))\subseteq\cdots\subseteq \pi (a_{t_{1}}\oplus (T_{C_{1}}\oplus (\cdots\oplus (T_{C_{n-1}}\oplus T_{C_{n}})\cdots )\subseteq \pi (T_{C_{0}}\oplus (T_{C_{1}}\oplus (\cdots \oplus (T_{C_{n-1}}\oplus T_{C_{n}})\cdots )\subseteq \pi ((U_{0}\oplus H)\oplus ((U_{1}\oplus H)\oplus (\cdots \oplus ((U_{n-1}\oplus H)\oplus (U_{n}\oplus H))\cdots )\subseteq \pi (U)$. By our assumption on $\mathcal{A}$, $\pi (T_{C_{0}}\oplus (T_{C_{1}}\oplus (\cdots \oplus (T_{C_{n-1}}\oplus T_{C_{n}})\cdots )\in\mathcal{A}''$ and $\pi (T_{C_{0}}\oplus (T_{C_{1}}\oplus (\cdots \oplus (T_{C_{n-1}}\oplus T_{C_{n}})\cdots )\subseteq B_{m_{k}}$ for some $k$. Consequently, $M\subseteq X\cap B_{m_{k}}$ and $M\subseteq\{x_{j,i}:j\leq k,i\in\omega\}$ by the
item (\expandafter{\romannumeral1}) of the construction of $X$. Since $\pi (0)$ is a unique cluster point of the set $\{x_{j,i}:j\leq k,i\in\omega\}$, the sequence $M$ cannot
converge to $t_{n}\neq \pi (0)$, which is a contradiction.

Therefore, $G/H$ is $snf$-countable. Moreover, it is well-known that a topological space $X$ is weakly first-countable if and only if $X$ is sequential and $snf$-countable. Then we obtain that $G/H$ is weakly first-countable, hence a first-countable space. We complete the proof.
\end{proof}

Therefore, we conclude the following results.

\begin{corollary}
Suppose that $G$ is a strongly topological gyrogroup, $H$ is a closed strong subgyrogroup of $G$ and $H$ is inner neutral, then the followings are equivalent.

$(1)$\, $G/H$ is first-countable;

$(2)$\, $G/H$ is a bisequential space;

$(3)$\, $G/H$ is a weakly first-countable space;

$(4)$\, $G/H$ is a $csf$-countable and sequential $\alpha_{7}$-space.
\end{corollary}

\section{Quotient spaces with locally compact strong subgyrogroups}

In this section, it is proved that if $G$ is a strongly topological gyrogroup and $H$ is a locally compact strong subgyrogroup of $G$, then there exists an open neighborhood $U$ of the identity element $0$ such that $\pi (\overline{U})$ is closed in $G/H$ and the restriction of $\pi $ to $\overline{U}$ is a perfect mapping from $\overline{U}$ onto the subspace $\pi (\overline{U})$, which implies that if $H$ is a locally compact strong subgyrogroup of a strongly topological gyrogroup $G$ and the quotient space $G/H$ has some nice properties, such as locally compact, locally countably compact, locally pseudocompact, etc., then $G$ also has the same properties.

\begin{proposition}\label{closed}
If $H$ is a locally compact subgyrogroup of a topological gyrogroup $G$, then $H$ is closed in $G$.
\end{proposition}

\begin{proof}
We show that $H=\overline{H}$ in $G$. By \cite[Proposition 7]{AW}, $\overline{H}$ is a subgyrogroup of $G$. Since $H$ is a dense locally compact subspace of $\overline{H}$, we know that $H$ is open in $\overline{H}$. Therefore, $H$ is closed, which means that $H=\overline{H}$. Therefore, $H$ is a closed subgyrogroup in $G$.
\end{proof}

\begin{lemma}\label{3subset}
Let $G$ be a topological gyrogroup. For each open neighborhood $U$ of $0$ and $x\in G$, there exists an open neighborhood $V$ of $0$ such that $V\oplus x\subseteq x\oplus U$.
\end{lemma}

\begin{proof}
For each open neighborhood $U$ of $0$ and $x\in G$, define $L_{\ominus x}(y):G\rightarrow G$ by $L_{\ominus x}(y)=\ominus x\oplus y$ and $R_{x}:G\rightarrow G$ by $R_{x}(y)=y\oplus x$. For each open neighborhood $U$ of $0$, we can find an open neighborhood $W$ of $x$ such that $L_{\ominus x}(W)=(\ominus x)\oplus W\subseteq U$, as $L_{\ominus x}$ is continuous. By the same method, we can find an open neighborhood $V$ of $0$ such that $R_{x}(V)=V\oplus x\subseteq W$. Then $(\ominus x)\oplus (V\oplus x)\subseteq U$, which means that $V\oplus x\subseteq x\oplus U$.
\end{proof}

\begin{theorem}\label{disjoint}
Let $G$ be a strongly topological gyrogroup with neighborhood base $\mathscr U$ of $0$, $F$ a compact subset of $G$, and $P$ a closed subset of $G$ such that $F\cap P=\emptyset$. Then there exists an open neighborhood $V$ of the identity element $0$ such that $(F\oplus V)\cap P=\emptyset$ and $(V\oplus F)\cap P=\emptyset$.
\end{theorem}

\begin{proof}
For each $x\in F$, we can find an open neighborhood $W_{x}$ of $0$ in $G$ such that $(x\oplus W_{x})\cap P=\emptyset$. Then choose an open neighborhood $O_{x}\in \mathscr U$ with $O_{x}\oplus O_{x}\subseteq W_{x}$. Since $F$ is compact and $F\subseteq \bigcup_{x\in F}\{x\oplus O_{x}\}$, we can find a finite subset $C\subseteq F$ such that $F\subseteq\bigcup_{x\in C}\{x\oplus O_{x}\}$. Put $Q=\bigcap_{x\in C}O_{x}$. For each $y\in F$, there exists $x\in C$ with $y\in x\oplus O_{x}$. Then $y\oplus Q\subseteq (x\oplus O_{x})\oplus Q\subseteq (x\oplus O_{x})\oplus O_{x}=\bigcup_{s,t\in O_{x}}\{(x\oplus s)\oplus t\}=\bigcup_{s,t\in O_{x}}\{x\oplus (s\oplus \mbox{gyr}[s,x](t))\}=x\oplus (O_{x}\oplus O_{x})\subseteq x\oplus W_{x}\subseteq G\setminus P$. Therefore, $(F\oplus Q)\cap P=\emptyset$.

On the other hand, choose $Q_{1}\in \mathscr U$ such that $Q_{1}\oplus Q_{1}\subseteq Q$. By Lemma \ref{3subset}, for each $y\in F$, we can find an open neighborhood $U_{y}$ of $0$ such that $U_{y}\oplus y\subseteq y\oplus Q_{1}$. Then $F\subseteq \bigcup_{y\in F}\{U_{y}\oplus y\}$. Since $F$ is compact, there is a finite set $E\subseteq F$ such that $F\subseteq \bigcup_{y\in E}\{U_{y}\oplus y\}$. Put $U=\bigcap_{y\in E}U_{y}$. For each $t\in F$, we can find $y\in E$ with $t\in U_{y}\oplus y$. Then
\begin{eqnarray}
U\oplus t&\subseteq&U\oplus (U_{y}\oplus y)\nonumber\\
&\subseteq&U\oplus (y\oplus Q_{1})\nonumber\\
&=&\bigcup\{u\oplus (y\oplus q):u\in U,q\in Q_{1}\}\nonumber\\
&=&\bigcup\{(u\oplus y)\oplus \mbox{gyr}[u,y](q):u\in U,q\in Q_{1}\}\nonumber\\
&=&(U\oplus y)\oplus Q_{1}\nonumber\\
&\subseteq&(U_{y}\oplus y)\oplus Q_{1}\nonumber\\
&\subseteq&(y\oplus Q_{1})\oplus Q_{1}\nonumber\\
&=&\bigcup\{(y\oplus p)\oplus q:p,q\in Q_{1}\}\nonumber\\
&=&\bigcup\{y\oplus (p\oplus \mbox{gyr}[p,y](q)):p,q\in Q_{1}\}\nonumber\\
&=&y\oplus (Q_{1}\oplus Q_{1})\nonumber\\
&\subseteq&y\oplus Q\nonumber\\
&\subseteq&G\setminus P\nonumber.
\end{eqnarray}
Therefore, $(U\oplus F)\cap P=\emptyset$. Finally, put $V=Q\cap U$, and we obtain that $(F\oplus V)\cap P=\emptyset$ and $(V\oplus F)\cap P=\emptyset$.
\end{proof}

\begin{lemma}\label{perfect1}
Let $G$ be a strongly topological gyrogroup and $H$ a locally compact strong subgyrogroup of $G$. Suppose that $P$ is a closed symmetric subset of $G$ such that $P$ contains an open neighborhood of $0$ in $G$, and $\overline{P\oplus (P\oplus P)}\cap H$ is compact. Then the restriction $f$ of $\pi$ to $P$ is a perfect mapping from $P$ onto the subspace $\pi(P)$ of $G/H$.
\end{lemma}

\begin{proof}
It is obvious that $f$ is continuous. Moreover, since $H$ is a locally compact subgyrogroup of $G$, we know that $H$ is closed in $G$ by Proposition \ref{closed}.

{\bf Claim 1.} $f^{-1}(f(a))$ is compact for each $a\in P$.

Indeed, from the definition of $f$, we have $f^{-1}(f(a))=(a\oplus H)\cap P$. By \cite[Proposition 3]{AW}, the left gyrotranslation is a homeomorphism, so the subspace $(a\oplus H)\cap P$ and $H\cap ((\ominus a)\oplus P)$ are homeomorphic, thus both of them are closed in $G$. From $\ominus a\in \ominus P=P$, it follows that $$H\cap ((\ominus a)\oplus P)\subseteq H\cap (P\oplus P)\subseteq \overline{P\oplus (P\oplus P)}\cap H.$$ Hence, $H\cap ((\ominus a)\oplus P)$ is compact and so is the set $f^{-1}f(a)$.

{\bf Claim 2.} $f$ is a closed mapping.

Let us fix any closed subset $M$ of $P$ and let $a$ be an any point of $P$ such that $f(a)\in \overline{f(M)}$. It suffices to show that $f(a)\in f(M)$.

Suppose on the contrary. Then $(a\oplus H)\cap (M\oplus H)\cap P=\emptyset$. Since $H$ is a strong subgyrogroup in $G$, $(a\oplus H)\oplus H=\bigcup_{x,y\in H}\{(a\oplus x)\oplus y\}=\bigcup_{x,y\in H}\{a\oplus (x\oplus \mbox{gyr}[x,a](y))\}=a\oplus (H\oplus H)=a\oplus H$. Then $(a\oplus H)\cap M\cap P=\emptyset$, and $(a\oplus H)\cap\overline{(P\oplus P)}\cap M=\emptyset$ since $M\subseteq P$. Obviously, $(a\oplus H)\cap \overline{(P\oplus P)}$ is compact. Since $M$ is a closed and disjoint from the compact subset $(a\oplus H)\cap \overline{(P\oplus P)}$, by Theorem \ref{disjoint}, there exists an open neighborhood $W\in \mathscr U$ such that $W\subseteq P$ and $(W\oplus ((a\oplus H)\cap \overline{(P\oplus P)}))\cap M=\emptyset $.

Since the quotient mapping $\pi $ is open by Lemma~\ref{t00000} and $W\oplus a$ is an open neighborhood of $a$, the set $\pi (W\oplus a)$ is an open neighborhood of $\pi (a)$ in $G/H$. Therefore, the set $\pi (W\oplus a)\cap \pi (M)\neq \emptyset $ and we can fix $m\in M$ and $y\in W$ such that $\pi (m)=\pi (y\oplus a)$, that is, $m\in (y\oplus a)\oplus H$. Then, $(y\oplus a)\oplus H=y\oplus (a\oplus gyr[a,y](H))=y\oplus (a\oplus H)$. Hence, there exists an $h\in H$ such that $a\oplus h=\ominus y\oplus m$. Since $\ominus y\in \ominus W=W\subseteq P$ and $m\in M\subseteq P$, we have that $a\oplus h=\ominus y\oplus m\in \overline{(P\oplus P)}$. In addition, $a\oplus h\in a\oplus H$. Hence, $a\oplus h\in ((a\oplus H)\cap \overline{(P\oplus P)})$ and $m\in W\oplus ((a\oplus H)\cap \overline{(P\oplus P)})$. Thus, $M\cap (W\oplus ((a\oplus H)\cap \overline{(P\oplus P)}))\neq \emptyset$, which is a contradiction.

Therefore, $f(a)\in f(M)$ and $f(M)$ is closed in $f(P)$. Then, since $a\in P$ is arbitrarily taken, we conclude that the mapping $f$ is perfect.
\end{proof}

The following result is very important in this paper.

\begin{theorem}\label{perfect2}
Let $G$ be a strongly topological gyrogroup and $H$ a locally compact strong subgyrogroup of $G$. Then there exists an open neighborhood $U$ of the identity element $0$ such that $\pi (\overline{U})$ is closed in $G/H$ and the restriction of $\pi $ to $\overline{U}$ is a perfect mapping from $\overline{U}$ onto the subspace $\pi (\overline{U})$.
\end{theorem}

\begin{proof}
Since $H$ is locally compact, we know that $H$ is closed in $G$ by Proposition \ref{closed} and we can find an open neighborhood $V$ of $0$ in $G$ such that $\overline{V\cap H}$ is compact. Since $G$ is regular, we can choose an open neighborhood $W$ of $0$ such that $\overline{W}\subseteq V$. As a closed subspace of the compact set $\overline{V\cap H}$, $\overline{W}\cap H$ is compact. Let $U_{0}$ be an arbitrary symmetric open neighborhood of $0$ such that $U_{0}\oplus (U_{0}\oplus U_{0})\subseteq W$. By the joint continuity, we have $\overline{U_{0}}\oplus (\overline{U_{0}}\oplus \overline{U_{0}})\subseteq \overline{U_{0}\oplus (U_{0}\oplus U_{0})}$. Then the set $P=\overline{U_{0}}$ satisfies all restrictions on $P$ in Lemma \ref{perfect1}. It follows from Lemma \ref{perfect1} that the restriction of $\pi $ to $P$ is a perfect mapping from $P$ onto the subspace $\pi (P)$.

It follows from Lemma \ref{t00000} that $\pi $ is an open mapping, the set $\pi (U_{0})$ is open in $G/H$. It follows from Theorem \ref{regular} that the space $G/H$ is regular, then we can find an open neighborhood $V_{0}$ of $\pi (0)$ in $G/H$ such that $\overline{V_{0}}\subseteq \pi (U_{0})$. Hence $U=\pi ^{-1}(V_{0})\cap U_{0}$ is an open neighborhood of $0$ contained in $P$ such that the restriction $f$ of $\pi$ to $\overline{U}$ is a perfect mapping from $\overline{U}$ onto the subspace $\pi (\overline{U})$. Furthermore, $\pi (\overline{U})$ is closed in $\pi (P)$, and $\pi (\overline{U})\subseteq \overline{V_{0}}\subseteq \pi (U_{0})\subseteq \pi (P)$. Then $\pi (\overline{U})$ is closed in $\overline{V_{0}}$, so that $\pi (\overline{U})$ is closed in $G/H$.
\end{proof}

\begin{corollary}\label{property}
Let $\mathscr P$ be a topological property preserved by preimages of spaces under perfect mappings (in the class of completely regular spaces) and also inherited by regular closed sets. Let $G$ be a strongly topological gyrogroup and $H$ a locally compact strong subgyrogroup of $G$. Then if the quotient space $G/H$ has the property $\mathscr P$, we can find an open neighborhood $U$ of the identity element $0$ such that $\overline{U}$ has the property $\mathscr P$.
\end{corollary}

As we all know, local compactness, countable compactness, pseudocompactness, the Lindel\"{o}f property, $\sigma $-compactness and $\check{\rm C}$ech-completeness are all inherited by regular closed sets and preserved by perfect preimages. Then the followings are clear by Corollary \ref{property}.

\begin{corollary}
Let $G$ be a strongly topological gyrogroup and $H$ a locally compact strong subgyrogroup of $G$. If the quotient space $G/H$ has some of the following properties:

\smallskip
(1) $G/H$ is locally compact;

\smallskip
(2) $G/H$ is locally countably compact;

\smallskip
(3) $G/H$ is locally pseudocompact;

\smallskip
(4) $G/H$ is locally $\sigma $-compact;

\smallskip
(5) $G/H$ is locally Lindel\"{o}f;

\smallskip
(6) $G/H$ is locally $\check{C}$ech-complete; and

\smallskip
(7) $G/H$ is locally realcompact,\\
\smallskip
then $G$ also has the same property.
\end{corollary}

Since paracompactness is inherited by regular closed sets and preserved by perfect preimages and it was proved in \cite[Theorem 4.6]{BL2} that every locally paracompact strongly topological gyrogroup is paracompact, the following result is trivial.

\begin{corollary}
Let $G$ be a strongly topological gyrogroup and $H$ a locally compact strong subgyrogroup of $G$. If the quotient space $G/H$ is locally paracompact, then $G$ is a paracompact space.
\end{corollary}

\begin{corollary}
Let $G$ be a strongly topological gyrogroup and $H$ a locally compact strong subgyrogroup of $G$. If the quotient space $G/H$ is a $k$-space, then $G$ is also a $k$-space.
\end{corollary}

\begin{proof}
Since the property of being a $k$-space is invariant under taking perfect preimages and a locally $k$-space is a $k$-space, see \cite[Section 3.3]{E}, it follows that $G$ is also a $k$-space.
\end{proof}

A topological gyrogroup is {\it feathered} if it contains a non-empty compact set $K$ of countable character in $G$.

\begin{lemma}\cite{BLL}\label{paracompact}
Let $G$ be a strongly topological gyrogroup. Then the followings are equivalent:

\smallskip
(1) $G$ is feathered,

\smallskip
(2) $G$ is a $p$-space, and

\smallskip
(3) $G$ is a paracompact $p$-space.
\end{lemma}

\begin{theorem}
Let $G$ be a strongly topological gyrogroup and $H$ a locally compact strong subgyrogroup of $G$. If the quotient space $G/H$ is a feathered space, then $G$ is a paracompact $p$-space.
\end{theorem}

\begin{proof}
By Theorem \ref{perfect2}, there exists an open neighborhood $U$ of the identity element $0$ in $G$ such that $\overline{U}$ is a preimage of a closed subset of $G/H$ under a perfect mapping. Moreover, since the class of feathered spaces is closed under taking closed subspaces, it follows from \cite[Proposition 4.3.36]{AA} that $\overline{U}$ is a feathered space. Therefore, $U$ contains a non-empty compact subspace $F$ with a countable base of neighborhoods in $G$, thus $G$ is a paracompact $p$-space by Lemma \ref{paracompact}.
\end{proof}

\section{Quotient spaces with locally compact and metrizable strong subgyrogroups}
In this section, we give some applications about Theorem \ref{perfect2} combining generalized metric properties. In particular, we assume that the strong subgyrogroup $H$ of a strongly topological gyrogroup $G$ is locally compact and metrizable.

\begin{definition}\cite{SF}
Let $X$ be a topological space. A space is called {\it strictly Fr\'echet-Urysohn at a point $x\in X$} if whenever $\{A_{n}\}_{n}$ is a sequence of subsets in $X$ and $x\in \bigcap _{n\in \mathbb{N}}\overline{A_{n}}$, there exists $x_{n}\in A_{n}$ for each $n\in \mathbb{N}$ such that the sequence $\{x_{n}\}_{n}$ converges to $x$. A space $X$ is called {\it strictly Fr\'echet-Urysohn} if it is strictly Fr\'echet-Urysohn at every point $x\in X$.
\end{definition}

\begin{lemma}\cite{LS}\label{i1}
Suppose that $X$ is a regular space, and that $f: X\rightarrow Y$ is a closed mapping. Suppose also that $b\in X$ is a $G_{\delta}$-point in the space $F=f^{-1}(f(b))$ (i.e., the singleton $\{b\}$ is a $G_{\delta}$-set in the space $F$) and $F$ is countably compact and strictly Fr\'echet-Urysohn at $b$. If the space $Y$ is strictly Fr\'echet-Urysohn at $f(b)$, then $X$ is strictly Fr\'echet-Urysohn at $b$.
\end{lemma}

\begin{theorem}
Let $G$ be a strongly topological gyrogroup and $H$ a locally compact metrizable strong subgyrogroup of $G$. If the quotient space $G/H$ is strictly Fr\'echet-Urysohn, then $G$ is also strictly Fr\'echet-Urysohn.
\end{theorem}

\begin{proof}
By Theorem \ref{perfect2}, there exists an open neighborhood $U$ of the identity element $0$ in $G$ such that $\pi |_{\overline{U}}:\overline{U}\rightarrow \pi (\overline{U})$ is a perfect mapping and $\pi (\overline{U})$ is closed in $G/H$.

Put $f=\pi |_{\overline{U}}:\overline{U}\rightarrow \pi (\overline{U})$. Then $f(\overline{U})=\pi (\overline{U})$ is strictly Fr\'echet-Urysohn. For each $b\in \overline{U}$, $f^{-1}(f(b))=\pi^{-1}(\pi (b))\cap \overline{U}=(b\oplus H)\cap \overline{U}$ is compact and metrizable. It follows from Lemma \ref{i1} that $\overline{U}$ is strictly Fr\'echet-Urysohn. Therefore, $G$ is locally strictly Fr\'echet-Urysohn and $G$ is strictly Fr\'echet-Urysohn.
\end{proof}

\begin{lemma}\cite[Proposition 4.7.18]{AA}\label{i}
Suppose that $X$ is a regular space, and that $f: X\rightarrow Y$ is a closed mapping. Suppose also that $b\in X$ is a $G_{\delta}$-point in the space $F=f^{-1}(f(b))$ (i.e., the singleton $\{b\}$ is a $G_{\delta}$-set in the space $F$) and $F$ is Fr\'echet-Urysohn at $b$. If the space $Y$ is strongly Fr\'echet-Urysohn, then $X$ is Fr\'echet-Urysohn at $b$.
\end{lemma}

\begin{theorem}
Let $G$ be a strongly topological gyrogroup and $H$ a locally compact metrizable strong subgyrogroup of $G$. If the quotient space $G/H$ is strongly Fr\'echet-Urysohn, then the space $G$ is also strongly Fr\'echet-Urysohn.
\end{theorem}

\begin{proof}
By Theorem \ref{perfect2}, there exists an open neighborhood $U$ of the identity element $0$ in $G$ such that $\pi |_{\overline{U}}:\overline{U}\rightarrow \pi (\overline{U})$ is a perfect mapping and $\pi (\overline{U})$ is closed in $G/H$.

Put $f=\pi |_{\overline{U}}:\overline{U}\rightarrow \pi (\overline{U})$. Then $f(\overline{U})=\pi (\overline{U})$ is strongly Fr\'echet-Urysohn. For each $b\in \overline{U}$, $f^{-1}(f(b))=\pi^{-1}(\pi (b))\cap \overline{U}=(b\oplus H)\cap \overline{U}$ is metrizable. Therefore, the singleton $\{b\}$ is a $G_{\delta}$-set in the space $f^{-1}(f(b))$. Moreover, since the quotient space $G/H$ is strongly Fr\'echet-Urysohn, the space $G$ is locally Fr\'echet-Urysohn by Lemma \ref{i}. Hence, $G$ is Fr\'echet-Urysohn. Furthermore, every Fr\'echet-Urysohn topological gyrogroup is strongly Fr\'echet-Urysohn by \cite[Corollary 5.2]{LF1}. So $G$ is strongly Fr\'echet-Urysohn.
\end{proof}

\begin{theorem}\label{4dl5}
Let $G$ be a strongly topological gyrogroup and $H$ a locally compact metrizable strong subgyrogroup of $G$. If the quotient space $G/H$ is sequential, then $G$ is also sequential.
\end{theorem}

\begin{proof}
By Theorem \ref{perfect2}, there exists an open neighborhood $U$ of the identity element $0$ in $G$ such that $\pi |_{\overline{U}}:\overline{U}\rightarrow \pi (\overline{U})$ is a perfect mapping and $\pi (\overline{U})$ is closed in $G/H$.

First, we show that if $\{x_{n}\}_{n}$ is a sequence in $\overline{U}$ such that $\{\pi (x_{n})\}_{n}$ is a convergent sequence in $\pi (\overline{U})$ and $x$ is an accumulation point of the sequence $\{x_{n}\}_{n}$, then there is a subsequence of $\{x_{n}\}_{n}$ which converges to $x$.

Since $\pi |_{\overline{U}}$ is perfect, every subsequence of $\{x_{n}\}_{n}$ has an accumulation point in $\overline{U}$. Put $F=\pi^{-1}(\pi (x))\cap \overline{U}$. Since $H$ is metrizable, $\pi^{-1}(\pi (x))=x\oplus H$ is also metrizable. Since every topological gyrogroup is regular, there exists a sequence $\{U_{k}\}_{k}$ of open subsets in $G$ such that $\overline{U_{k+1}}\subseteq U_{k}$ for each $k\in \mathbb{N}$ and $\{x\}=F\cap \bigcap_{k\in \mathbb{N}}U_{k}$. Choose a subsequence $\{x_{n_{k}}\}_{k}$ of $\{x_{n}\}_{n}$ such that $x_{n_{k}}\in U_{k}$ for each $k\in \mathbb{N}$. For an arbitrary accumulation point $p$ of a subsequence of the sequence $\{x_{n_{k}}\}_{k}$, we have $\pi (p)=\pi (x)$ and $p\in \bigcap _{k\in \mathbb{N}}\overline{U_{k}}$. Thus $p=x$. Therefore, $x$ is the unique accumulation point of every subsequence of $\{x_{n_{k}}\}_{k}$, proving that $x_{n_{k}}\rightarrow x$.

Then choose an open neighborhood $V$ of $0$ such that $\overline{V}\subseteq U$ and we show that $\overline{V}$ is a sequential subspace.

Suppose that $\overline{V}$ is not a sequential subspace, so we can find a non-closed and sequentially closed subset $A$ of $\overline{V}$. Then there exists a point $x$ such that $x\in cl_{\overline{V}}(A)\setminus A$. It is clear that $cl_{\overline{V}}(A)=\overline{A}$. Let $f=\pi |_{\overline{V}}:\overline{V}\rightarrow \pi (\overline{V})$ and $B=A\cap f^{-1}(f(x))$. Since $B$ is a closed subset of $A$, $B$ is sequentially closed. Moreover, the fiber $f^{-1}(f(x))=(\pi^{-1}(\pi(x)))\cap \overline{V}$ is sequential, so $B$ is closed in $\overline{V}$. Since $x\not \in B$, there exists an open neighborhood $W$ of $x$ in $\overline{V}$ such that $\overline{W}\cap B=\emptyset$. Let $C=\overline{W}\cap A$, then $C$ is also sequentially closed as a closed subset of $A$ and $x\in \overline{C}\setminus C$. Therefore, $C\cap f^{-1}(f(x))=\overline{W}\cap B=\emptyset$, then $f(x)\in \overline{f(C)}\setminus f(C)$. So $f(C)=\pi (C)$ is not closed in $\pi (\overline{V})$. However, this is impossible, as it is easy to verify that the image of each sequentially closed subset of $\overline{V}$ is closed in $\pi (\overline{V})$.

Indeed, let $C$ be sequentially closed in $\overline{V}$ and $\{y_{n}\}_{n}$  a sequence in $\pi (C)$ such that $y_{n}\rightarrow y$ in $\pi (\overline{V})$. Choose $x_{n}\in C$ with $\pi (x_{n})=y_{n}$ for each $n\in \mathbb{N}$. Since every subsequence of the sequence $\{x_{n}\}_{n}$ has an accumulation point, there exist a point $x\in \pi^{-1}(y)$ and a subsequence $\{x_{n_{k}}\}_{k}$ of $\{x_{n}\}_{n}$ such that $x_{n_{k}}\rightarrow x$. Since $C$ is sequentially closed, we obtain $x\in C$ and $y\in \pi (C)$. Therefore, $\pi (C)$ is sequentially closed in $\pi (\overline{V})$. Since $\pi |_{\overline{U}}:\overline{U}\rightarrow \pi (\overline{U})$ is a closed mapping and $\pi (\overline{U})$ is closed in $G/H$, $\pi (\overline{V})$ is closed in $G/H$. Since $G/H$ is sequential, $\pi (\overline{V})$ is also sequential and then $\pi (C)$ is closed in $\pi (\overline{V})$.

Since $G$ is homogeneous and $\overline{V}$ is a sequential subspace, we conclude that $G$ is a locally sequential space. Thus, $G$ is a sequential space.
\end{proof}

However, for the property of Fr\'echet-Urysohn, we do not know whether it has the similar result. Therefore, we pose the following question.

\begin{question}
Let $G$ be a strongly topological gyrogroup and $H$ a locally compact metrizable strong subgyrogroup of $G$. If the quotient space $G/H$ is Fr\'echet-Urysohn, is $G$ also Fr\'echet-Urysohn?
\end{question}

\begin{theorem}\label{4dl6}
Let $G$ be a strongly topological gyrogroup and $H$ a locally compact metrizable strong subgyrogroup of $G$. If the quotient space $G/H$ has property $\mathcal{P}$, where $\mathcal{P}$ is a topological property. Then the space $G$ is locally in $\mathcal{P}$ if $\mathcal{P}$ satisfies the following:

\smallskip
(1) $\mathcal{P}$ is closed hereditary;

\smallskip
(2) $\mathcal{P}$ contains point $G_{\delta}$-property, and

\smallskip
(3) let $f:X\rightarrow Y$ be a perfect mapping, if $X$ has $G_{\delta}$-diagonal and $Y$ is $\mathcal{P}$, then $X$ is $\mathcal{P}$.

\end{theorem}

\begin{proof}
By the hypothesis, since $G/H$ is in $\mathcal{P}$ and $\mathcal{P}$ contains point $G_{\delta}$-property, $\{H\}$ is a $G_{\delta}$-subset in $G/H$, that is, there exists a sequence $\{V_{n}:n\in \mathbb{N}\}$ of open sets in $G/H$ such that $\{H\}=\bigcap_{n\in \mathbb{N}}V_{n}$. Therefore, $H=\bigcap _{n\in \mathbb{N}}\pi^{-1}(V_{n})$. Since $H$ is a metrizable strong subgyrogroup of $G$, there is a family $\{W_{n}:n\in \mathbb{N}\}$ of open neighborhoods of the identity element $0$ such that $\{W_{n}\cap H:n\in \mathbb{N}\}$ is an open countable neighborhood base in $H$. Hence, $$\{0\}=\bigcap_{n\in \mathbb{N}}(W_{n}\cap H)=\bigcap_{n\in \mathbb{N}}(W_{n}\cap \pi^{-1}(V_{n})).$$ Then $G$ has point $G_{\delta}$-property. It follows from \cite{BL1} that every strongly topological gyrogroup with countable pseudocharacter is submetrizable. So $G$ has $G_{\delta}$-diagonal.

By Theorem \ref{perfect2}, there is an open neighborhood $U$ of the identity element $0$ in $G$ such that $\pi |_{\overline{U}}:\overline{U}\rightarrow \pi (\overline{U})$ is a perfect mapping and $\pi (\overline{U})$ is closed in $G/H$. Then by (1) and (3), the subspace $\overline{U}$ is in $\mathcal{P}$. Therefore, $G$ is locally in $\mathcal{P}$.
\end{proof}

It is well-known that all stratifiable spaces, semi-stratifiable spaces and $\sigma$-spaces satisfy the conditions in Theorem \ref{4dl6}, respectively. Moreover, it was claimed in \cite{BLX} and \cite{LS2} that if a strongly topological gyrogroup $G$ has point $G_{\delta}$-property, then $G$ has a $KG$-sequence and if $f:X\rightarrow Y$ is a perfect map and $Y$ is a $k$-semistratifiable space, then $X$ is a $k$-semistratifiable space if and only if $X$ has a $KG$-sequence. Therefore, the following corollary is obtained.

\begin{corollary}
Let $G$ be a strongly topological gyrogroup and $H$ a locally compact metrizable strong subgyrogroup of $G$. If the quotient space $G/H$ is a stratifiable space (semi-stratifiable space, $k$-semistratifiable, $\sigma$-space), then $G$ is a local stratifiable space (semi-stratifiable space, $k$-semistratifiable, $\sigma$-space).
\end{corollary}

Finally, we pose the following questions.

\begin{question}
Let $G$ be a strongly topological gyrogroup and $H$ a closed strong subgyrogroup of $G$. Is the quotient space $G/H$ completely regular?
\end{question}

\begin{question}
Let $\mathcal{P}$ be any class of topological spaces which is closed hereditary and closed under locally finite unions of closed sets. Is every strongly topological gyrogroup which is locally in $\mathcal{P}$ in $\mathcal{P}$ ?
\end{question}

\section{Quotient spaces with closed first-countable and separable strong subgyrogroups}

In this section, we study the quotient space $G/H$ with some generalized metric properties, where $G$ is a strongly topological gyrogroup and $H$ is a closed first-countable and separable strong subgyrogroup of $G$. In particular, we prove that if the quotient space $G/H$ is an $\aleph_{0}$-space, then $G$ is an $\aleph_{0}$-space; if the quotient space $G/H$ is a cosmic space, then $G$ is also a cosmic space; if the quotient space $G/H$ has a star-countable $cs$-network or star-countable $wcs^{*}$-network, then $G$ also has a star-countable $cs$-network or star-countable $wcs^{*}$-network, respectively.

\begin{definition}\cite{GMT, LS1}
Let $\mathcal{P}$ be a family of subsets of a topological space $X$.

1. $\mathcal{P}$ is called a {\it k-network} for $X$ if whenever $K\subseteq U$ with $K$ compact and $U$ open in $X$, there exists a finite family $\mathcal{P}^{'}\subseteq \mathcal{P}$ such that $K\subseteq \bigcup \mathcal{P}^{'}\subseteq U$.

2. $\mathcal{P}$ is called a {\it $wcs^{*}$-network} for $X$ if, given a sequence $\{x_{n}\}_{n}$ converging to a point $x$ in $X$ and a neighborhood $U$ of $x$ in $X$, there exists a subsequence $\{x_{n_{i}}\}_{i}$ of the sequence $\{x_{n}\}_{n}$ such that $\{x_{n_{i}}:i\in \mathbb{N}\}\subseteq P\subseteq U$ for some $P\in \mathcal{P}$.

\end{definition}

\begin{definition}\cite{ME1} Let $X$ be a topological space.

1. $X$ is called {\it cosmic} if $X$ is a regular space with a countable network.

2. $X$ is called an $\aleph_{0}$-space if it is a regular space with a countable $k$-network.
\end{definition}

It was claimed in \cite{linbook} that every base is a $k$-network and a $cs$-network for a topological space, and every $k$-network or every $cs$-network is a $wcs^{*}$-network for a topological space, but the converse does not hold. Moreover, a space $X$ has a countable $cs$-network if and only if $X$ has a countable $k$-network if and only if $X$ has a countable $wcs^{*}$-network, see \cite{LS}. Therefore, it is natural that a topological space is an $\aleph_{0}$-space if and only if it is a regular space with a countable $cs$-network. Moreover, every $\aleph_{0}$-space is a cosmic space and every cosmic space is a paracompact, separable space.

The following lemmas are necessary.

\begin{lemma}\label{t00002}\cite{BLX}
Suppose that $G$ is a topological gyrogroup and $H$ is a closed and separable $L$-subgyrogroup of $G$. If $Y$ is a separable subset of $G/H$, $\pi ^{-1}(Y)$ is also separable in $G$.
\end{lemma}

\begin{lemma}\cite{BL2}\label{t00003}
Every locally paracompact strongly topological gyrogroup is paracompact.
\end{lemma}

\begin{lemma}\label{t00004}\cite{BD}
Every star-countable family $\mathcal{P}$ of subsets of a topological space $X$ can be expressed as $\mathcal{P}=\bigcup \{\mathcal{P} _{\alpha}:\alpha \in \Lambda\}$, where each subfamily $\mathcal{P}_{\alpha}$ is countable and $(\bigcup \mathcal{P} _{\alpha})\cap (\bigcup \mathcal{P} _{\beta})=\emptyset$ whenever $\alpha \not =\beta$.
\end{lemma}

\begin{theorem}\label{3dl3}
Let $G$ be a strongly topological gyrogroup and $H$ a closed first-countable and separable strong subgyrogroup of $G$. If the quotient space $G/H$ is a local $\aleph_{0}$-space, then $G$ is a topological sum of $\aleph_{0}$-subspace.
\end{theorem}

\begin{proof}
Let $G$ be a strongly topological gyrogroup with a symmetric neighborhood base $\mathscr U$ at $0$. Since the quotient space $G/H$ is a local $\aleph_{0}$-space, we can find an open neighborhood $Y$ of $H$ in $G/H$ such that $Y$ has a countable $cs$-network. Put $X=\pi ^{-1}(Y)$. By Lemma \ref{t00000}, the natural homomorphism $\pi$ from $G$ onto $G/H$ is an open and continuous mapping, so $X$ is an open neighborhood of the identity element $0$ in $G$. Since $Y$ is an $\aleph_{0}$-space and each $\aleph_{0}$-space is separable, it follows from Lemma \ref{t00002} that $X$ is separable. Therefore, there is countable subset $B=\{b_{m}:m\in \mathbb{N}\}$ of $X$ such that $\overline{B}=X$.

Since $H$ is first-countable, there exists a countable family $\{V_{n}:n\in \mathbb{N}\}\subseteq \mathscr U$ of open symmetric neighborhoods of $0$ in $G$ such that $V_{n+1}\oplus (V_{n+1}\oplus V_{n+1})\subseteq V_{n}\subseteq X$ for each $n\in \mathbb{N}$ and the family $\{V_{n}\cap H:n\in \mathbb{N}\}$ is a local base at $0$ for $H$. Since $Y$ is an $\aleph_{0}$-space, there is a countable $cs$-network $\{P_{k}:k\in \mathbb{N}\}$ for $Y$.

{\bf Claim 1.} $X$ is an $\aleph_{0}$-space.

Put $\mathcal{F}=\{\pi ^{-1}(P_{k})\cap (b_{m}\oplus V_{n}):k,m,n\in \mathbb{N}\}$. Then $\mathcal{F}$ is a countable family of subsets of $X$. Suppose that $\{x_{i}\}_{i}$ is a sequence converging to a point $x$ in $X$ and $U$ be a neighborhood of $x$ in $X$. Then $U$ is also a neighborhood of $x$ in $G$. Let $V$ be an open neighborhood of $0$ in $G$ such that $x\oplus (V\oplus V)\subseteq U$. Since $\{V_{n}\cap H:n\in \mathbb{N}\}$ is a local base at $0$ for $H$, there is $n\in \mathbb{N}$ such that $V_{n}\cap H\subseteq V\cap H$. Moreover, $(x\oplus V_{n+1})\cap X$ is a non-empty open subset of $X$ and $\overline{B}=X$, whence $B\cap (x\oplus V_{n+1})\not =\emptyset$. Therefore, there exists $b_{m}\in B$ such that $b_{m}\in x\oplus V_{n+1}$. Furthermore, $(x\oplus V_{n+1})\cap (x\oplus V)$ is an open neighborhood of $x$ and $\pi :G\rightarrow G/H$ is an open mapping, so $\pi ((x\oplus V_{n+1})\cap (x\oplus V))$ is an open neighborhood of $\pi (x)$ in the space $Y$ and the sequence $\{\pi (x_{i})\}_{i}$ converges to $\pi (x)$ in $Y$. It is obtained that $$\{\pi (x)\}\cup \{\pi (x_{i}):i\geq i_{0}\}\subseteq P_{k}\subseteq \pi ((x\oplus V_{n+1})\cap (x\oplus V)) \mbox{ for some } i_{0},k\in \mathbb{N}.$$ By the left cancellation law of Lemma \ref{a}, it is easy to verify that $(x\oplus V_{n+1})\cap (x\oplus V)=x\oplus (V_{n+1}\cap V)$.
Therefore, for an arbitrary $z\in \pi ^{-1}(P_{k})\cap (b_{m}\oplus V_{n+1})$, $\pi (z)\in P_{k}\subseteq \pi (x\oplus (V_{n+1}\cap V))$. Since $z\in (x\oplus (V_{n+1}\cap V))\oplus H$, and $H$ is a strong subgyrogroup, then
$$z\in (x\oplus (V_{n+1}\cap V))\oplus H=\bigcup_{t\in V_{n+1}\cap V}\{(x\oplus t)\oplus H\}=\bigcup_{t\in V_{n+1}\cap V}\{x\oplus (t\oplus \mbox{gyr}[t,x](H))\}=x\oplus ((V_{n+1}\cap V)\oplus H).$$

Therefore, $\ominus x\oplus z\in (V_{n+1}\cap V)\oplus H$. Moreover, since $z\in b_{m}\oplus V_{n+1}$ and $b_{m}\in x\oplus V_{n+1}$, it follows that
\begin{eqnarray}
z&\in&(x\oplus V_{n+1})\oplus V_{n+1}\nonumber\\
&=&\bigcup_{u,v\in V_{n+1}}\{(x\oplus u)\oplus v\}\nonumber\\
&=&\bigcup_{u,v\in V_{n+1}}\{x\oplus (u\oplus \mbox{gyr}[u,x](v))\}\nonumber\\
&=&x\oplus (V_{n+1}\oplus V_{n+1}).\nonumber
\end{eqnarray}
So, $(\ominus x)\oplus z\in V_{n+1}\oplus V_{n+1}$. Hence, $(\ominus x)\oplus z\in ((V_{n+1}\cap V)\oplus H)\cap (V_{n+1}\oplus V_{n+1})$. There exist $a\in (V_{n+1}\cap V),~h\in H$ and $u_{3},v_{3}\in V_{n+1}$ such that $(\ominus x)\oplus z=a\oplus h=u_{3}\oplus v_{3}$, whence $h=(\ominus a)\oplus (u_{3}\oplus v_{3})\in V_{n+1}\oplus (V_{n+1}\oplus V_{n+1})\subseteq V_{n}$. Therefore, $(\ominus x)\oplus z\in (V_{n+1}\cap V)\oplus (V_{n}\cap H)$, and consequently, $z\in x\oplus ((V_{n+1}\cap V)\oplus (V_{n}\cap H))\subseteq x\oplus (V\oplus V)\subseteq U$. Thus, we obtain that $\pi ^{-1}(P_{k})\cap (b_{m}\oplus V_{n+1})\subseteq U$.

Since $b_{m}\in x\oplus V_{n+1}$, there is $u\in V_{n+1}$ such that $b_{m}=x\oplus u$, whence
\begin{eqnarray}
x&=&(x\oplus u)\oplus \mbox{gyr}[x,u](\ominus u)\nonumber\\
&=&b_{m}\oplus \mbox{gyr}[x,u](\ominus u)\nonumber\\
&\in&b_{m}\oplus \mbox{gyr}[x,u](V_{n+1})\nonumber\\
&=&b_{m}\oplus V_{n+1}.\nonumber
\end{eqnarray}
Therefore, there exists $i_{1}\geq i_{0}$ such that $x_{i}\in b_{m}\oplus V_{n+1}$ when $i\geq i_{1}$, whence $\{x\}\cup \{x_{i}:i\geq i_{1}\}\subseteq \pi ^{-1}(P_{k})\cap (b_{m}\oplus V_{n+1})$. Thus $\mathcal{F}$ is a countable $cs$-network for $X$, and hence $X$ is an $\aleph_{0}$-space.

Since $G$ is homogeneous, it is clear that $G$ is a local $\aleph_{0}$-space. Therefore, $G$ is a locally paracompact space. Furthermore, every locally paracompact strongly topological gyrogroup is paracompact by Lemma \ref{t00003}, so $G$ is paracompact. Let $\mathcal{A}$ be an open cover of $G$ by $\aleph_{0}$-subspace. Because the property of being an $\aleph_{0}$-space is hereditary, we can assume that $\mathcal{A}$ is locally finite in $G$ by the paracompactness of $G$. Moreover, as every point-countable family of open subsets in a separable space is countable, the family $\mathcal{A}$ is star-countable. Then $\mathcal{A}=\bigcup \{\mathcal{B}_{\alpha}:\alpha \in \Lambda\}$ by Lemma \ref{t00004}, where each subfamily $\mathcal{B}_{\alpha}$ is countable and $(\bigcup \mathcal{B}_{\alpha})\cap (\bigcup \mathcal{B}_{\beta})=\emptyset$ whenever $\alpha \not =\beta$. Set $X_{\alpha}=\bigcup \mathcal{B}_{\alpha}$ for each $\alpha \in \Lambda$. Then $G=\bigoplus_{\alpha \in \Lambda}X_{\alpha}$.

{\bf Claim 2.} $X_{\alpha}$ is an $\aleph_{0}$-subspace for each $\alpha \in \Lambda$.

Put $\mathcal{B}_{\alpha}=\{B_{\alpha ,n}:n\in \mathbb{N}\}$, where each $B_{\alpha ,n}$ is an open $\aleph_{0}$-subspace of $G$, and put $\mathcal{P}_{\alpha}=\bigcup _{n\in \mathbb{N}}\mathcal{P}_{\alpha ,n}$, where $\mathcal{P}_{\alpha ,n}$ is a countable $cs$-network for the $\aleph_{0}$-space $B_{\alpha ,n}$ for each $n\in \mathbb{N}$. Then $\mathcal{P}_{\alpha}$ is a countable $cs$-network for $X_{\alpha}$. Thus, $X_{\alpha}$ is an $\aleph_{0}$-space.

In conclusion, $G$ is a topological sum of $\aleph_{0}$-subspaces.
\end{proof}

\begin{corollary}
Let $G$ be a strongly topological gyrogroup and $H$ a closed first-countable and separable strong subgyrogroup of $G$. If the quotient space $G/H$ is an $\aleph_{0}$-space, $G$ is also an $\aleph_{0}$-space.
\end{corollary}

By the similar proof of Theorem \ref{3dl3}, the following result is obvious.

\begin{theorem}
Let $G$ be a strongly topological gyrogroup and $H$ a closed first-countable and separable strong subgyrogroup of $G$. If the quotient space $G/H$ is a locally cosmic space, then $G$ is a topological sum of cosmic subspaces.
\end{theorem}

\begin{corollary}
Let $G$ be a strongly topological gyrogroup and $H$ a closed first-countable and separable strong subgyrogroup of $G$. If the quotient space $G/H$ is a cosmic space, $G$ is also a cosmic space.
\end{corollary}

\begin{theorem}\label{3dl2}
Let $G$ be a strongly topological gyrogroup and $H$ a closed first-countable and separable strong subgyrogroup of $G$. If the quotient space $G/H$ has a star-countable $cs$-network, $G$ also has a star-countable $cs$-network.
\end{theorem}

\begin{proof}
Let $\mathscr{U}$ be a symmetric neighborhood base at $0$ such that $gyr[x, y](U)=U$ for any $x, y\in G$ and $U\in\mathscr{U}$. Since the subgyrogroup $H$ of $G$ is first-countable at the identity element $0$ of $G$, there exists a countable family $\{V_{n}:n\in \mathbb{N}\}\subseteq \mathscr U$ such that $(V_{n+1}\oplus (V_{n+1}\oplus V_{n+1}))\subseteq V_{n}$ for each $n\in \mathbb{N}$ and the family $\{V_{n}\cap H:n\in \mathbb{N}\}$ is a local base at $0$ for $H$.

Let $\mathcal{P}=\{P_{\alpha}:\alpha \in \Lambda\}$ be a star-countable $cs$-network for the space $G/H$. For each $\alpha \in \Lambda$, the family $\{P_{\alpha}\cap P_{\beta}:\beta \in \Lambda\}$ is a countable $wcs^{*}$-network for $P_{\alpha}$. Therefore, $P_{\alpha}$ is a cosmic space, and $P_{\alpha}$ is separable. Then it follows from Lemma \ref{t00002} that $\pi^{-1}(P_{\alpha})$ is separable. We can find a countable subset $B_{\alpha}=\{b_{\alpha ,m}:m\in \mathbb{N}\}$ of $\pi^{-1}(P_{\alpha})$ such that $\overline{B_{\alpha}}=\pi^{-1}(P_{\alpha})$.

Put $$\mathcal{F}=\{\pi ^{-1}(P_{\alpha})\cap (b_{\alpha ,m}\oplus V_{n}):\alpha \in \Lambda ,~~and~~m,n\in \mathbb{N}\}.$$ Then $\mathcal{F}$ is a star-countable family of $G$.

{\bf Claim.} $\mathcal{F}$ is a $cs$-network for $G$.

Let $\{x_{i}\}_{i}$ be a sequence converging to a point $x$ in $G$ and let $U$ be a neighborhood of $x$ in $G$. Choose an open neighborhood $V$ of $0$ in $G$ such that $(x\oplus (V\oplus V))\subseteq U$. Since $\{V_{n}\cap H:n\in \mathbb{N}\}$ is a local base at $0$ for $H$, there exists $n\in \mathbb{N}$ such that $V_{n}\cap H\subseteq V\cap H$. Since $\pi :G\rightarrow G/H$ is an open and continuous mapping, there are $i_{0}\in \mathbb{N}$ and $\alpha \in \Lambda$ such that $\{\pi (x)\}\cup \{\pi (x_{i}):i\geq i_{0}\}\subseteq P_{\alpha}\subseteq \pi ((x\oplus V_{n+1})\cap (x\oplus V))$. Since $x\in \pi^{-1}(P_{\alpha})$, $(x\oplus V_{n+1})\cap \pi^{-1}(P_{\alpha})$ is non-empty and open in the subspace $\pi^{-1}(P_{\alpha})$. Moreover, since $\overline{B_{\alpha}}=\pi^{-1}(P_{\alpha})$, there exists $m\in \mathbb{N}$ such that $b_{\alpha ,m}\in x\oplus V_{n+1}$.

For an arbitrary $z\in \pi^{-1}(P_{\alpha})\cap (b_{\alpha ,m}\oplus V_{n+1})$, $\pi (z)\in P_{\alpha}\subseteq \pi ((x\oplus V_{n+1})\cap (x\oplus V))=\pi (x\oplus (V_{n+1}\cap V))$. Then, $z\in x\oplus ((V_{n+1}\cap V)\oplus H)$ since $H$ is a strong subgyrogroup. Since $z\in b_{\alpha ,m}\oplus V_{n+1}$ and $b_{\alpha ,m}\in x\oplus V_{n+1}$, we have
\begin{eqnarray}
z&\in&(x\oplus V_{n+1})\oplus V_{n+1}\nonumber\\
&=&\bigcup_{u,v\in V_{n+1}}\{(x\oplus u)\oplus v\}\nonumber\\
&=&\bigcup_{u,v\in V_{n+1}}\{x\oplus (u\oplus \mbox{gyr}[u,x](v))\}\nonumber\\
&=&x\oplus (V_{n+1}\oplus V_{n+1}).\nonumber
\end{eqnarray}
Then, $(\ominus x)\oplus z\in V_{n+1}\oplus V_{n+1}$. Hence, $(\ominus x)\oplus z\in ((V_{n+1}\cap V)\oplus H)\cap (V_{n+1}\oplus V_{n+1})$. Therefore, there exist $a\in (V_{n+1}\cap V),~h\in H$ and $u_{1},u_{2}\in V_{n+1}$ such that $(\ominus x)\oplus z=a\oplus h=u_{1}\oplus u_{2}$, whence $h=(\ominus a)\oplus (u_{1}\oplus u_{2})\in V_{n+1}\oplus (V_{n+1}\oplus V_{n+1})\subseteq V_{n}$. It follows that $(\ominus x)\oplus z\in (V_{n+1}\cap V)\oplus (V_{n}\cap H)$. Thus $z\in x\oplus ((V_{n+1}\cap V)\oplus (V_{n}\cap H))\subseteq x\oplus (V\oplus V)\subseteq U$. Hence, $\pi^{-1}(P_{\alpha})\cap (b_{\alpha ,m}\oplus V_{n+1})\subseteq U$.

Since $b_{\alpha ,m}\in x\oplus V_{n+1}$, there is $u_{3}\in V_{n+1}$ such that $b_{\alpha ,m}=x\oplus u_{3}$. Thus,
\begin{eqnarray}
x&=&(x\oplus u_{3})\oplus \mbox{gyr}[x,u_{3}](\ominus u_{3})\nonumber\\
&=&b_{\alpha ,m}\oplus \mbox{gyr}[x,u_{3}](\ominus u_{3})\nonumber\\
&\in&b_{\alpha ,m}\oplus \mbox{gyr}[x,u_{3}](V_{n+1})\nonumber\\
&=&b_{\alpha ,m}\oplus V_{n+1}.\nonumber
\end{eqnarray}
Therefore, there exists $i_{1}\geq i_{0}$ such that $x_{i}\in b_{\alpha ,m}\oplus V_{n+1}$ whenever $i\geq i_{1}$, whence $\{x\}\cup \{x_{i}:i\geq i_{1}\}\subseteq \pi ^{-1}(P_{\alpha})\cap (b_{\alpha ,m}\oplus V_{n+1})$.

In conclusion, $G$ has a star-countable $cs$-network.
\end{proof}

\begin{theorem}\label{3dl4}
Let $G$ be a strongly topological gyrogroup and $H$ a closed first-countable and separable strong subgyrogroup of $G$. If the quotient space $G/H$ has a star-countable $wcs^{*}$-network, $G$ also has a star-countable $wcs^{*}$-network.
\end{theorem}

\begin{proof}
Let $\mathscr{U}$ be a symmetric neighborhood base at $0$ such that $gyr[x, y](U)=U$ for any $x, y\in G$ and $U\in\mathscr{U}$. Since the subgyrogroup $H$ of $G$ is first-countable at the identity element $0$ of $G$, there exists a countable family $\{V_{n}:n\in \mathbb{N}\}\subseteq \mathscr U$ in $G$ such that $(V_{n+1}\oplus (V_{n+1}\oplus V_{n+1}))\subseteq V_{n}$ for each $n\in \mathbb{N}$ and the family $\{V_{n}\cap H:n\in \mathbb{N}\}$ is a local base at $0$ for $H$.

We construct $\mathcal{P}$ and $\mathcal{F}$ by the same way in Theorem \ref{3dl2}, and we show that $\mathcal{F}$ is a $wcs^{*}$-network for $G$.

Let $\{x_{i}\}_{i}$ be a sequence converging to a point $x$ in $G$ and $U$ be a neighborhood of $x$ in $G$. Choose an open neighborhood $V$ of $0$ in $G$ such that $(x\oplus (V\oplus V))\subseteq U$. Since $\{V_{n}\cap H:n\in \mathbb{N}\}$ is a local base at $0$ for $H$, there exists $n\in \mathbb{N}$ such that $V_{n}\cap H\subseteq V\cap H$. Since $\mathcal{P}$ is a $wcs^{*}$-network for $G/H$, there exists a subsequence $\{\pi (x_{i_{j}})\}_{j}$ of the sequence $\{\pi (x_{i})\}_{i}$ such that $\{\pi (x_{i_{j}}):j\in \mathbb{N}\}\subseteq P_{\alpha}\subseteq \pi ((x\oplus V_{n+1})\cap (x\oplus V))$ for some $\alpha \in \Lambda$. As the sequence $\{x_{i}\}_{i}$ converges to $x$, we have some $x_{i_{j}}\in x\oplus V_{n+2}$ for each $j\in \mathbb{N}$. Furthermore, since $x_{i_{1}}\in \pi ^{-1}(P_{\alpha})$, $(x_{i_{1}}\oplus V_{n+2})\cap \pi^{-1}(P_{\alpha})$ is non-empty and open in $\pi^{-1}(P_{\alpha})$. Then it follows from $\overline{B_{\alpha}}=\pi^{-1}(P_{\alpha})$ that there exists $m\in \mathbb{N}$ such that $b_{\alpha ,m}\in x_{i_{1}}\oplus V_{n+2}$. Then
\begin{eqnarray}
b_{\alpha ,m}&\in&x_{i_{1}}\oplus V_{n+2}\nonumber\\
&\subseteq&(x\oplus V_{n+2})\oplus V_{n+2}\nonumber\\
&=&\bigcup_{u,v\in V_{n+2}}\{(x\oplus u)\oplus v\}\nonumber\\
&=&\bigcup_{u,v\in V_{n+2}}\{x\oplus (u\oplus \mbox{gyr}[u,x](v))\}\nonumber\\
&=&x\oplus (V_{n+2}\oplus V_{n+2}).\nonumber
\end{eqnarray}

Moreover, it is proved in Theorem \ref{3dl2} that $\pi^{-1}(P_{\alpha})\cap (b_{\alpha ,m}\oplus V_{n+1})\subseteq U$.

In conclusion, $G$ has a star-countable $wcs^{*}$-network.
\end{proof}

\end{document}